\documentclass[14pt]{amsart}

\usepackage{cases}
\usepackage{amsmath}
\usepackage{amsfonts}

\usepackage{bm}

\usepackage{amsfonts,amsmath,amssymb,amscd,bbm,amsthm,mathrsfs,dsfont}
\usepackage{mathrsfs}
\usepackage{pb-diagram}
\usepackage{amssymb}
\usepackage{xypic}

\newtheorem{Theorem}{Theorem}[section]
\newtheorem{Lemma}[Theorem]{Lemma}

\newtheorem{Definition}[Theorem]{Definition}
\newtheorem{Corollary}[Theorem]{Corollary}
\newtheorem{Proposition}[Theorem]{Proposition}
\newtheorem{Example}[Theorem]{Example}
\newtheorem{Remark}[Theorem]{Remark}
\newtheorem{Conjecture}[Theorem]{Conjecture}
\newtheorem{Problem}[Theorem]{Problem}

\date{version of \today}

 \normalbaselines
\title
[The enough $g$-pairs property and denominator vectors of cluster algebras]
{The enough $g$-pairs property and denominator vectors\\ of cluster algebras}

\author{Peigen Cao $\;\;\;\;\;\;$ Fang Li $\;\;\;\;\;\;$}
\address{Peigen Cao
\newline Department
of Mathematics, Zhejiang University (Yuquan Campus), Hangzhou, Zhejiang
310027,  P.R.China}
\email{peigencao@126.com}

\address{Fang Li
\newline Department
of Mathematics, Zhejiang University (Yuquan Campus), Hangzhou, Zhejiang
310027, P.R.China}
\email{fangli@zju.edu.cn}

\begin{document}

\renewcommand{\thefootnote}{\alph{footnote}}

\setcounter{footnote}{-1}  
\renewcommand{\thefootnote}{\alph{footnote}}
\setcounter{footnote}{-1} 

\begin{abstract} In this paper, we introduce the enough $g$-pairs property for a principal coefficients cluster algebra, which can be understood as a strong version of the sign-coherence of the $G$-matrices.  Then we prove that any skew-symmetrizable principal coefficients cluster algebra has the enough $g$-pairs property. As an application, we prove the positivity of denominator vectors for any skew-symmetrizable cluster algebra. In fact, we give  complete answers to some long standing conjectures  on denominator vectors of cluster variables (see Conjecture \ref{conjecture} below), which are proposed by Fomin and Zelevinsky in [Compos. Math. 143(2007), 112-164].

In addition, we prove that the seeds whose clusters contain particular cluster variables form a connected  subgraph of the exchange graph of this cluster algebra. Lastly, a criterion to distinguish whether particular cluster variables belong to one common cluster is given.
\end{abstract}

\maketitle
\bigskip

\section{introduction}

Cluster algebras were introduced by Fomin and Zelevinsky in \cite{FZ}. The motivation was to create a common framework for phenomena occurring in connection
with total positivity and canonical bases. A cluster algebra $\mathcal A(\mathcal S)$  is a subalgebra of an ambient field $\mathcal F$ generated by certain combinatorially defined generators (i.e., {\em cluster variables}), which are grouped into overlapping {\em clusters}.  Fomin and  Zelevinsky proved the Laurent phenomenon, that is,  for any given cluster ${\bf x}_{t_0}=\{x_{1;t_0},\cdots,x_{n;t_0}\}$ of $\mathcal A(\mathcal S)$, each cluster variable $x_{i;t}$ of $\mathcal A(\mathcal S)$ can be written as
 $$x_{i;t}=\frac{f(x_{1;t_0},\cdots,x_{n;t_0})}{x_{1;t_0}^{d_1}\cdots x_{n;t_0}^{d_n}},$$
where $f$ is a polynomial which is not divisible by any variable $x_{j;t_0}$. The vector $d^{t_0}(x_{i;t})=(d_1,\cdots,d_n)^{\top}$ is called  the {\em denominator vector} (or say, {\em $d$-vector} for short) of the cluster variable $x_{i;t}$ with respect to ${\bf x}_{t_0}$.  Fomin and Zelevinsky in \cite{FZ} conjectured that the coefficients in $f$ are positive, which has been affirm by Lee and Schiffler in \cite{LS} for skew-symmetric cluster algebras, by Gross {\em et al.} in \cite{GHKK} for skew-symmetrizable cluster algebras, by Davison in \cite{D} for skew-symmetric quantum cluster algebras. The classic positivity conjecture is for numerator. In fact, there is also the correspond positivity conjecture for denominator, which is about denominator vectors.

\begin{Conjecture}(Fomin and Zelevinsky \cite[Conjecture 7.4]{FZ3})\label{conjecture}

(i). (Positivity of $d$-vectors) If $x_{i;t}\notin {\bf x}_{t_0}$, then $d^{t_0}(x_{i;t})$ is a non-negative vector, i.e., $d^{t_0}(x_{i;t})\in\mathbb N^n$;

(ii). Each component $d_k$ depends only on $x_{i;t}$ and $x_{k;t_0}$, not on the clusters containing $x_{k;t_0}$;

(iii). $d_k=0$ if and only if there is a cluster containing both $x_{i;t}$ and $x_{k;t_0}$.
\end{Conjecture}

\begin{Remark}
We found that each statement in Conjecture \ref{conjecture} is very useful.

(a).  The statement (i) makes us provide a new proof of linear independence of cluster monomials (see Proposition \ref{thmproper}).

(b). The statement (ii) makes it possible to give a well-defined function, called compatible degree, on the set of cluster variables for any skew-symmetrizable cluster algebra  (see Definition \ref{defdegree}).

(c). The statement (iii) inspires us to give an answer to Problem \ref{problem} below (see Theorem \ref{thmcompatible}).
\end{Remark}

Compared with other conjectures, such as the positivity conjecture of cluster variables, the sign-coherence conjecture of $C$-matrices and $G$-matrices, the linear independence conjecture in cluster algebras (referring  to \cite{FZ,FZ3} for these conjectures),  only a little of progress on  Conjecture \ref{conjecture} has been made  over the past years.

 For cluster algebras of finite type, Conjecture \ref{conjecture} is affirmed  in \cite{CP}.
 For cluster algebras from surfaces, the components of $d$-vectors are  interpreted as certain modified intersection numbers (see \cite{FST}). This fact gives the affirmation of Conjecture \ref{conjecture} for such cluster algebras.

 For acyclic skew-symmetric cluster algebras, the $d$-vectors with respect to an acyclic seed, can be interpreted as the dimension vectors of indecomposable exceptional modules in a module category over a hereditary algebra (see \cite{CK}), thus Conjecture \ref{conjecture} (i) holds for the $d$-vectors with respect to an acyclic seed.
Recently, the positive answer to Conjecture \ref{conjecture} (i) for  {\em skew-symmetric cluster algebras} has been given  by us in \cite{CL1}. Generally speaking, Conjecture \ref{conjecture} remains largely open now.

 Why are  $d$-vectors  difficult to study? One of the reasons is that it is hard to find  a right model to explain  the $d$-vectors with respect to a general initial seed,  and this is also the reason making  $d$-vectors so mysterious in cluster algebras.

Now we list other conjectures or problems which will be involved in this paper.

\begin{Conjecture} (Fomin and Zelevinsky \cite[Conjecture 4.14(3)]{FZ2})\label{conjecture2}
In a cluster algebra, the seeds whose clusters contain particular cluster variables form a connected  subgraph of the exchange graph of this cluster algebra.
\end{Conjecture}

\begin{Problem}\label{problem}
Let $x_1,\cdots,x_p$ be some different cluster variables of $\mathcal A(\mathcal S)$. Find a criterion to distinguish
whether $x_1,\cdots,x_p$  belong to one common cluster  of $\mathcal A(\mathcal S)$.
\end{Problem}

Conjecture \ref{conjecture2} has positive answers for acyclic skew-symmetric cluster algebras and cluster algebras from surface (see \cite{CK,FST}).
 Problem \ref{problem}  has a nice answer for cluster algebras from surface, or cluster algebras admitting a categorification (see \cite{FST,CK}). In such cases, cluster variables are in bijection with tagged arcs of the marked
 surface or reachable indecomposable rigid objects in corresponding cluster category and clusters are in bijection with triangulations of the marked surface or reachable cluster tilting objects.

 In present paper a positive answer to Conjecture \ref{conjecture2} and a  criterion for Problem \ref{problem} are given  for any skew-symmetrizable cluster algebra (see Theorem \ref{corconnected} and Theorem \ref{thmcompatible}).

Now we talk about the strategy in this paper. The start point of this paper is the positivity of $d$-vectors. In order to prove the positivity of $d$-vectors, we introduce the definition of the enough $g$-pairs property for a cluster algebra with principal coefficients. What is the enough $g$-pairs property and why do we introduce it?
   In order to study the properties of the Laurent expansion of a cluster variable $x_{i;t}$ with respect to an initial cluster ${\bf x}_{t_0}$, we hope to find a new cluster ${\bf x}_{t^\prime}$ such that
    the properties that we focus on are preserved  when we turn to study the  Laurent expansion of $x_{i;t}$ with respect to this new cluster ${\bf x}_{t^\prime}$ and these properties for the  Laurent expansion of $x_{i;t}$ with respect to this new cluster ${\bf x}_{t^\prime}$ is  easier to study. Based on this philosophy, the enough $g$-pairs property for a cluster algebra with principal coefficients is introduced.

    The next thing is  which cluster algebra with principal coefficients has the enough $g$-pairs property. Thanks to \cite[Theorem 33]{M} by Muller on scattering diagram, we prove that any skew-symmetrizable cluster algebra with principal coefficients has the enough $g$-pairs property (see Theorem \ref{thmenough}). This is a very interesting  thing to us, because the sign-coherence of $G$-matrices is a direct result of the enough $g$-pairs property. So the enough $g$-pairs property can be understood as a strong version of the  sign-coherence of $G$-matrices.

    There are two main results on $g$-pairs. One is called the existence theorem for $g$-pairs (Theorem \ref{thmenough}), which says that any skew-symmetrizable cluster algebra with principal coefficients has the enough $g$-pairs property. The other is called the uniqueness theorem for $g$-pairs (Theorem \ref{thmunique}), which says that the $g$-pairs is unique if we fix a cluster and a subset of $\{1,\cdots,n\}$.  The existence theorem is important in this paper.

    We introduce the enough $g$-pairs property  to prove the positivity of $d$-vectors at first. Quickly, we found that  answers to Conjecture \ref{conjecture}, Conjecture \ref{conjecture2} and  Problem \ref{problem}   can be given  after  the enough $g$-pairs property  is introduced.

This paper is organized as follows.

In Section 2 some basic definitions and notations are introduced.
In Section 3 we recall some properties $g$-vectors of  cluster algebras.

In Section 4  we briefly review the scattering diagrams, using \cite{M} as a reference. Then we prove the existence theorem for $g$-pairs, i.e., any skew-symmetrizable cluster algebra with principal coefficients at $t_0$ has  the enough $g$-pairs property  (Theorem \ref{thmenough}). In addition, we explain why the sign-coherence of $G$-matrices is a direct result of  the enough $g$-pairs property. In Section 5 we prove the uniqueness theorem for the $g$-pairs (Theorem \ref{thmunique}).

In Section 6, we give ours answers to  Conjecture \ref{conjecture} and Conjecture \ref{conjecture2} (see Theorem \ref{thmanswer} and Theorem \ref{corconnected}).
In Section 7, we give an answer to Problem \ref{problem} in Theorem \ref{thmcompatible}.

 The following diagram gives  the logical dependence among the proofs of some main theorems in this paper.
$$
\begin{array}{ccc}
\xymatrix{Theorem\; \ref{thmenough}\; (\text{The enough $g$-pair property})\ar[r]\ar[d]\ar[rd]&Theorem \;\ref{corconnected}\; (\text{Answer to Conjecture \ref{conjecture2}})\ar[d]
\\
Theorem\; \ref{thmcompatible}\;(\text{Answer to Problem \ref{problem}})&Theorem \;\ref{thmanswer}\;(\text{Answer to Conjecture \ref{conjecture}}\ar[l])
}
\end{array}
$$
\vspace{5mm}

\section{Preliminaries}

\subsection{Cluster algebras}
Recall that $(\mathbb P, \oplus, \cdot)$ is a {\bf semifield } if $(\mathbb P,  \cdot)$ is an abelian multiplicative group endowed with a binary operation of auxiliary addition $\oplus$ which is commutative, associative, and distributive with respect to the multiplication $\cdot$ in $\mathbb P$.

Let $Trop(y_1,\cdots,y_m)$ be a free abelian group generated by $\{y_1,\cdots,y_m\}$. We define the addition $\oplus$ in $Trop(y_1,\cdots,y_m)$ by $\prod\limits_{i}y_i^{a_i}\oplus\prod\limits_{i}y_i^{b_i}=\prod\limits_{i}y_i^{min(a_i,b_i)}$, then $(Trop(y_1,\cdots,y_m), \oplus)$ is a semifield, which is called a {\bf tropical semifield}.

The multiplicative group of any semifield $\mathbb P$ is torsion-free \cite{FZ}, hence its group ring $\mathbb Z\mathbb P$ is a domain.
We take an ambient field $\mathcal F$  to be the field of rational functions in $n$ independent variables with coefficients in $\mathbb Z\mathbb P$.

An integer matrix $B_{n\times n}=(b_{ij})$ is  called  {\bf skew-symmetrizable} if there is a positive integer diagonal matrix $S$ such that $SB$ is skew-symmetric, where $S$ is said to be a {\bf skew-symmetrizer} of $B$.

\begin{Definition}
A {\bf  seed} in $\mathcal F$ is a triplet $({\bf x},{\bf y},B)$ such that

(i)  ${\bf x}=\{x_1,\cdots, x_n\}$ is a transcendence basis for $\mathcal F$ over Frac($\mathbb {ZP}$). ${\bf x}$ is called the  {\bf cluster} of  $({\bf x},{\bf y},B)$ and $x_1\cdots,x_n$  are called {\bf cluster variables}.

(ii) ${\bf y}=\{y_1,\cdots,y_n\}$ is a subset of  $\mathbb P$, where $y_1,\cdots,y_n$ are called {\bf coefficients}.

(iii) $B=(b_{ij})$ is a skew-symmetrizable matrix, called an {\bf exchange matrix}.
\end{Definition}

Let $({\bf x},{\bf y}, B)$ be a  seed in $\mathcal F$, one can associate {\em binomials} $F_1,\cdots,F_n$ defined by
\begin{eqnarray}
F_j=\frac{y_j}{1\oplus y_j}\prod\limits_{b_{ij}>0}x_i^{b_{ij}}+\frac{y_j}{1\oplus y_j}\prod\limits_{b_{ij}<0}x_i^{-b_{ij}}.\nonumber
\end{eqnarray}
$F_1,\cdots,F_n$ are called the {\bf exchange polynomials} of $({\bf x},{\bf y}, B)$. Clearly, $x_i\nmid F_j$ for any $i$ and $j$. Denote by ${\bf F}=\{F_1,\cdots,F_n\}$ the collection of  exchange polynomials of $({\bf x},{\bf y}, B)$.

\begin{Definition}\label{defmutation}
Let $({\bf x},{\bf y},B)$ be a  seed in $\mathcal F$, and $F_1,\cdots, F_n$ be the exchange polynomials of  $({\bf x},{\bf y},B)$. Define the {\bf mutation}  of  $({\bf x},{\bf y},B)$ in the direction $k\in\{1,\cdots,n\}$ as a new triple $\mu_k({\bf x},{\bf y},B)=( {\bf x}^{\prime},  {\bf y}^{\prime},  B^{\prime})$ in $\mathcal F$
given by
\begin{eqnarray}
b_{ij}^{\prime}&=&\begin{cases}-b_{ij}~,& i=k\text{ or } j=k;\\ b_{ij}+sgn(b_{ik})max(b_{ik}b_{kj},0)~,&otherwise.\end{cases}\nonumber\\
x_i^{\prime}&=&\begin{cases}x_i~,&\text{if } i\neq k\\ F_k/x_k,~& \text{if }i=k.\end{cases}\text{ and }
 y_i^{\prime}=\begin{cases} y_k^{-1}~,& i=k\\ y_iy_k^{max(b_{ki},0)}(1\bigoplus y_k)^{-b_{ki}}~,&otherwise.
\end{cases}.\nonumber
\end{eqnarray}
\end{Definition}
It can be seen that $\mu_k({\bf x},{\bf y},B)$ is also a  seed and $\mu_k(\mu_k({\bf x},{\bf y},B))=({\bf x},{\bf y},B)$.

\begin{Definition}\label{defcpattern}
  A {\bf cluster pattern} $\mathcal S$ is an assignment of a  seed  $({\bf x}_t,{\bf y}_t,B_t)$ to every vertex $t$ of the $n$-regular tree $\mathbb T_n$, such that for any edge $t^{~\underline{\quad k \quad}}~ t^{\prime},~({\bf x}_{t^{\prime}},{\bf y}_{t^{\prime}},B_{t^{\prime}})=\mu_k({\bf x}_t,{\bf y}_t,B_t)$.
\end{Definition}
We always denote by ${\bf x}_t=\{x_{1;t},\cdots, x_{n;t}\},~ {\bf y}_t=\{y_{1;t},\cdots, y_{n;t}\}, ~B_t=(b_{ij}^t).$
\begin{Definition} Let  $\mathcal S$ be a cluster pattern,   the {\bf cluster algebra} $\mathcal A(\mathcal S)$  associated with the given cluster pattern $\mathcal S$ is the $\mathbb {ZP}$-subalgebra of the field $\mathcal F$ generated by all cluster variables of  $\mathcal S$.
 \end{Definition}
 \begin{itemize}
 \item If  $\mathcal S$ is cluster pattern with the coefficients in $Trop(y_1,\cdots,y_m)$, the corresponding cluster algebra $\mathcal A(\mathcal S)$ is said to be a cluster algebra of {\bf geometric type}.
 \item If  $\mathcal S$ is cluster pattern  with the coefficients in $Trop(y_1,\cdots,y_n)$ and there exists a seed $({\bf x}_{t_0},{\bf y}_{t_0},B_{t_0})$ such that $y_{i;t_0}=y_i$ for $i=1,\cdots,n$, then  the corresponding cluster algebra  $\mathcal A(\mathcal S)$ is called a {\bf cluster algebra with principal coefficients at $t_0$}.
\end{itemize}
\begin{Definition}
(i). In  a cluster algebra $\mathcal A(\mathcal S)$, two seeds $({\bf x}_{t_1},{\bf y}_{t_1},B_{t_1})$ and  $({\bf x}_{t_2},{\bf y}_{t_2},B_{t_2})$ are called {\bf equivalent} if there exists a permutation $\sigma$ of $\{1,\cdots,n\}$ such that $x_{i;t_1}=x_{\sigma(i);t_2}$, $y_{i;t_1}=y_{\sigma(i);t_2}$ and $b_{ij}^{t_1}=b_{\sigma(i)\sigma(j)}^{t_2}$ for any $i$ and $j$.

(ii). The exchange graph $\Gamma(\mathcal S)$ of a cluster algebra $\mathcal A(\mathcal S)$ is a graph with vertices corresponding to equivalent classes of seeds of  $\mathcal A(\mathcal S)$ and edges corresponding to mutations.
\end{Definition}

Let ${\bf x}_t$ be a cluster of a cluster algebra $\mathcal A(\mathcal S)$, denote by ${\bf x}_t^{\bf a}:=\prod\limits_{i=1}^nx_{i;t}^{a_i}$ for ${\bf a}\in \mathbb Z^n$, which is a Laurent monomial in ${\bf x}_t$. If ${\bf a}\in\mathbb N^n$, then ${\bf x}_t^{\bf a}$ is called a {\bf cluster monomial} in ${\bf x}_t$. If ${\bf a}\in \mathbb Z^n\backslash \mathbb N^n$, then ${\bf x}_t^{\bf a}$ is called a {\bf proper Laurent monomial} in ${\bf x}_t$.

\begin{Theorem} \label{thmLP} Let $\mathcal A(\mathcal S)$ be a skew-symmetrizable cluster algebra, and $({\bf x}_{t_0},{\bf y}_{t_0},B_{t_0})$ be a  seed of $\mathcal A(\mathcal S)$.

(i).( \cite[Theorem 3.1]{FZ}, Laurent phenomenon)
 Any cluster variable $x_{i;t}$ of $\mathcal A(\mathcal S)$ is a $\mathbb {ZP}$-linear combination of  Laurent monomials in ${\bf x}_{t_0}$.

(ii). (\cite{FZ1}, sharpen Laurent phenomenon) If $\mathcal A(\mathcal S)$ is a cluster algebra of geometric type with coefficients in $Trop(y_1,\cdots,y_m)$, then any cluster variable $x_{i;t}$ is a
$\mathbb Z[y_1,\cdots,y_m]$-linear combination of  Laurent monomials in ${\bf x}_{t_0}$.

(iii). (\cite{GHKK}, positive  Laurent phenomenon)  Any cluster variable $x_{i;t}$ of $\mathcal A(\mathcal S)$ is a $\mathbb {NP}$-linear combination of Laurent monomials in ${\bf x}_{t_0}$.
\end{Theorem}

Let $\mathcal A(\mathcal S)$ be a cluster algebra with principal coefficients at $t_0$, one can give a $\mathbb Z^n$-grading of $\mathbb Z[x_{1;t_0}^{\pm1},\cdots,x_{n;t_0}^{\pm1},y_1,\cdots,y_n]$ as follows:
$$deg(x_{i;t_0})={\bf e}_i,~deg(y_j)=-{\bf b}_j,$$
where ${\bf e}_i$ is the $i$-th column vector of $I_n$, and ${\bf b}_j$ is the $j$-th column vector of $B_{t_0}$, $i,j=1,2,\cdots,n$. As shown in \cite{FZ3} every cluster variable $x_{i;t}$ of $\mathcal A(\mathcal S)$ is homogeneous with respect to this $\mathbb Z^n$-grading. The {\bf $g$-vector} $g(x_{i;t})$ of a cluster variable $x_{i;t}$ is defined to be its degree with respect to the $\mathbb Z^n$-grading and we write $g(x_{i;t})=(g_{1i}^t,~g_{2i}^t,~\cdots,~g_{ni}^t)^{\top}\in\mathbb Z^n$. Let ${\bf x}_t$ be a cluster of $\mathcal A(\mathcal S)$,  the matrix $G_t=(g(x_{1;t}),\cdots,g(x_{n;t}))$ is called the {\bf $G$-matrix} of ${\bf x}_t$.

Clearly, any Laurent monomial ${\bf x}_t^{\bf a}=\prod\limits_{i=1}^{n}x_{i;t}^{a_i}\in\mathcal F$ is also homogeneous. And the degree of  ${\bf x}_t^{\bf a}$ is $G_t{\bf a}=:g({\bf x}_t^{\bf a})$, which is called the {\em $g$-vector of ${\bf x}_t^{\bf a}$}.

\begin{Theorem}(Sign-coherence \cite{GHKK})\label{thmfg} Let $\mathcal A(\mathcal S)$ be a skew-symmetrizable cluster algebra with principal coefficients at $t_0$, and ${\bf x}_t$ be a cluster of $\mathcal A(\mathcal S)$. Then

(i) any two nonzero entries of the $G$-matrix $G_t$ of ${\bf x}_t$ in the same row have the same sign.

(ii)  each cluster variable $x_{i;t}$  has the form of
\begin{equation}\label{gvectorsign}
x_{i;t}={\bf x}_{t_0}^{g(x_{i;t})}(1+\sum\limits_{0\neq {\bf v}\in\mathbb N^n}c_{\bf v}{\bf y}^{\bf v}{\bf x}_{t_0}^{B_{t_0}{\bf v}})\nonumber,
\end{equation}
where $g(x_{i;t})$ is the {\bf g}-vector of $x_{i;t}$ , ${\bf x}_{t_0}^{g(x_{i;t})}=\prod\limits_{j=1}^nx_{i;t_0}^{g_{ji}^t}$, ${\bf y}^{\bf v}=\prod\limits_{j=1}^ny_i^{v_i}$, and $c_{\bf v}\geq 0$.
\end{Theorem}

\subsection{The $d$-vectors}
  Let $\mathcal A(\mathcal S)$ be a skew-symmetrizable cluster algebra, and $({\bf x}_{t_0}, {\bf y}_{t_0}, B_{t_0})$ be a seed of  $\mathcal A(\mathcal S)$. By Laurent phenomenon,  any cluster variable $x$ of $\mathcal A(\mathcal S)$ has the form of $x=\sum\limits_{{\bf v}\in V} c_{\bf v}{\bf x}_{t_0}^{\bf v}$, where $V$ is a subset of $\mathbb Z^n$, $0\neq c_{\bf v}\in \mathbb {ZP}$. Let $-d_{j}$ be the minimal exponent of $x_{j;t_0}$ appearing in the expansion $x=\sum\limits_{{\bf v}\in V} c_{\bf v}{\bf x}_{t_0}^{\bf v}$. Then $x$ has the form of
\begin{eqnarray}
\label{eqd}x=\frac{f(x_{1;t_0},\cdots,x_{n;t_0})}{x_{1;t_0}^{d_1}\cdots x_{n;t_0}^{d_n}},
\end{eqnarray}
where $f\in\mathbb {ZP}[x_{1;t_0},\cdots,x_{n;t_0}]$ with  $x_{j;t_0}\nmid f$ for $j=1,\cdots,n$.
The vector $d^{t_0}(x) = (d_1,\cdots, d_n)^{\top}$ is called the {\bf denominator vector} (briefly, {\bf $d$-vector})  of the cluster variable $x$ with respect to ${\bf x}_{t_0}$. For a cluster ${\bf x}_t$, the matrix $D_t^{t_0}=(d^{t_0}(x_{1;t}),\cdots,d^{t_0}(x_{n;t}))$ is called the {\bf $D$-matrix} of ${\bf x}_t$ with respect to ${\bf x}_{t_0}$.
Let ${\bf x}_t^{\bf v}$ be a cluster monomial in ${\bf x}_t$, it is easy to see that the $d$-vector of ${\bf x}_t^{\bf v}$ with respect to ${\bf x}_{t_0}$ is $d^{t_0}({\bf x}_t^{\bf v})=D_t^{t_0}{\bf v}$.

\begin{Proposition}\label{prodvectors}
(\cite[Proposition 2.5]{RS}, \cite[Proposition 2.7]{CL1})
 Let $\mathcal A(\mathcal S)$ be a cluster algebra, $x$ be a cluster variable, and $({\bf x}_t,{\bf y}_t,B_t)$, $({\bf x}_{t^\prime}, {\bf y}_{t^\prime},B_{t^\prime})$ be two seeds of  $\mathcal A(\mathcal S)$
with $({\bf x}_{t^\prime}, {\bf y}_{t^\prime},B_{t^\prime})=\mu_k({\bf x}_t,{\bf y}_t,B_t)$. Suppose that $d^t(x)=(d_1,\cdots,d_n)^{\top}$, $d^{t^\prime}(x)=(d_1^\prime,\cdots,d_n^\prime)^{\top}$ are the $d$-vectors of $x$ with respective to ${\bf x}_t$ and ${\bf x}_{t^\prime}$ respectively, then $d_i=d_i^{\prime}$ for $i\neq k$.
\end{Proposition}

\begin{Proposition}((7.7) of \cite{FZ3})\label{prodrec}
$D_t^{t_0}=(d_{ij}^t)$ is uniquely determined by the initial condition $D_{t_0}^{t_0}=-I_n$, together with the recurrence relations:
\begin{equation}\label{mutationD}
(D_{t^{\prime}}^{t_0})_{ij}=\begin{cases}d_{ij}^t  & \text{if } j\neq k;\\ -d_{ik}^t+max\{\sum\limits_{b_{lk}^t>0}d_{il}^t b_{lk}^t, \sum\limits_{b_{lk}^t<0} -d_{il}^tb_{lk}^t\}  &\text{if } j=k.\end{cases}\nonumber
 \end{equation}
 for any $t,t'\in\mathbb T_n$ with edge $t^{~\underline{\quad k \quad}} ~t^{\prime}$.
\end{Proposition}

From the above proposition, we know that the notion of $d$-vectors is independent of the choice of coefficient system. So when studying the $d$-vector $d^{t_0}(x)$ of a cluster variable $x$, we can focus on the cluster algebras with principal coefficients at $t_0$.

 \vspace{5mm}

\section{On $g$-vectors of cluster algebras with principal coefficients}
In this subsection, we give some results  on $g$-vectors for cluster algebras with principal coefficients, which will be  used in the sequel.

\begin{Theorem}\label{detgthm}
Let $\mathcal A(\mathcal S)$ be a skew-symmetrizable cluster algebra with principal coefficients at $t_0$, and ${\bf x}_t, {\bf x}_{t^{\prime}}$ be any two clusters of $\mathcal A(\mathcal S)$, then

(i). The Laurent expansion of  $x_{j;t^{\prime}}$ with respect to ${\bf x}_t$ has the following form

 $$x_{j;t^{\prime}}={\bf x}_{t}^{{\bf r}_{j;t}^{t^{\prime}}}(1+\sum\limits_{0\neq {\bf v}\in\mathbb N^n,~~{\bf u}\in\mathbb Z^n}c_{\bf v}{\bf y}^{\bf v}{\bf x}_t^{\bf u}),$$
where $c_{\bf v}\geq 0$ and ${\bf r}_{j;t}^{t^{\prime}}$ satisfies $g(x_{j;t^{\prime}})=G_t{\bf r}_{j;t}^{t^{\prime}}$.

(ii) $G_{t^{\prime}}=G_tR_{t}^{t^{\prime}}$, where $R_{t}^{t^{\prime}}=({\bf r}_{1;t}^{t^{\prime}},\cdots,{\bf r}_{n;t}^{t^{\prime}})$ . In particular, $det G_t=\pm1$.

\end{Theorem}

\begin{proof}
By  the positivity of  Laurent phenomenon, the Laurent expansion of  $x_{j;t^{\prime}}$ with respect to ${\bf x}_t$ has the form of

$$x_{j;t^{\prime}}=\sum\limits_{p\in P}\lambda_{p,j}{\bf x}_{t}^{{\bf r}_{p,j;t}^{t^\prime}}+\sum\limits_{0\neq {\bf v}\in\mathbb N^n,~~{\bf u}\in\mathbb Z^n}c_{\bf v}{\bf y}^{\bf v}{\bf x}_t^{\bf u},$$
where $\lambda_{p,j}>0$ for $p\in P$, and $c_{\bf v}\geq0$.
Thanks to Theorem \ref{thmfg} (ii), we have that
\begin{eqnarray}
{\bf x}_{t_0}^{g(x_{j;t^\prime})}=x_{j;t^{\prime}}|_{y_1=\cdots=y_n=0}=\sum\limits_{p\in P}\lambda_{p,j}{\bf x}_{t}^{{\bf r}_{p,j;t}^{t^\prime}}|_{y_1=\cdots=y_n=0}=
\sum\limits_{p\in P}\lambda_{p,j}{\bf x}_{t_0}^{G_t{\bf r}_{p,j;t}^{t^\prime}}, \lambda_{p,j}> 0.\nonumber
\end{eqnarray}
The above equality holds if and only if $P$ has exactly one element (say $p_0$), and
\begin{eqnarray}\label{coreqa}
\lambda_{p_0,j}=1,\; g(x_{j;t^\prime})=G_t{\bf r}_{p_0,j;t}^{t^\prime}.\nonumber
\end{eqnarray}
So we can assume that the  expansion of  $x_{j;t^{\prime}}$ with respect to ${\bf x}_t$ is in the following form
$$x_{j;t^{\prime}}={\bf x}_{t}^{{\bf r}_{j;t}^{t^\prime}}+\sum\limits_{0\neq {\bf v}\in\mathbb N^n,~~{\bf u}\in\mathbb Z^n}c_{\bf v}{\bf y}^{\bf v}{\bf x}_t^{\bf u},\text{ with }c_{\bf v}\geq 0,$$
i.e., $x_{i;t^{\prime}}$ has the following form $$x_{j;t^{\prime}}={\bf x}_{t}^{{\bf r}_{j;t}^{t^{\prime}}}(1+\sum\limits_{0\neq {\bf v}\in\mathbb N^n,~~{\bf u}\in\mathbb Z^n}c_{\bf v}{\bf y}^{\bf v}{\bf x}_t^{\bf u}),$$
where $c_{\bf v}\geq 0$ and ${\bf r}_{j;t}^{t^\prime}$ satisfies $g(x_{j;t^\prime})=G_t{\bf r}_{j;t}^{t^\prime}$, $j=1,\cdots,n$.
So we have $G_{t^{\prime}}=G_tR_{t}^{t^{\prime}}$.
Take $t^{\prime}=t_0$, we have $G_tR_{t}^{t_0}=G_{t_0}=I_n$. Because $G_t$ and $R_{t}^{t_0}$ are integer matrices, we get $det G_t=\pm1$.
\end{proof}

The following theorem is  known from \cite{GHKK} for skew-symmetrizable cluster algberas. For skew-symmetric cluster algebras, one can also refer to \cite{DWZ}. Here we would like to provide an elementary proof for it. This new proof  depends only on the positivity of cluster variables and  the property that each $F$-polynomial has constant term $1$, i.e., Theorem \ref{thmfg} (ii), which is equivalent to the sign-coherence of $C$-matrices by  \cite[Proposition 5.6]{FZ3}.

\begin{Theorem}\label{thmmonomial}
Let $\mathcal A(\mathcal S)$ be a skew-symmetrizable cluster algebra with principal coefficients at $t_0$, and ${\bf x}_{t_1}^{\bf a}, {\bf x}_{t_2}^{\bf f}$ be  two cluster monomials of $\mathcal A(\mathcal S)$ . If ${\bf x}_{t_1}^{\bf a}$ and ${\bf x}_{t_2}^{\bf f}$ have the same $g$-vector, i.e., $G_{t_1}{\bf a}=G_{t_2}{\bf f}$, then ${\bf x}_{t_1}^{\bf a}={\bf x}_{t_2}^{\bf f}$.

\end{Theorem}
\begin{proof}
By Theorem \ref{detgthm} (i),  the Laurent expansion of the cluster monomial  ${\bf x}_{t_1}^{\bf a}$ with respect to ${\bf x}_{t_2}$  has the following form
\begin{equation}\label{equ12}
{\bf x}_{t_1}^{\bf a}={\bf x}_{t_2}^{{\bf b}}(1+\sum\limits_{0\neq {\bf v}\in\mathbb N^n,~~{\bf u}\in\mathbb Z^n}c_{\bf v}{\bf y}^{\bf v}{\bf x}_{t_2}^{\bf u}),\;\;c_{\bf v}\geq 0.
\end{equation}
Thanks to Theorem \ref{thmfg} (ii), we obtain that
\begin{eqnarray}
{\bf x}_{t_0}^{G_{t_1}{\bf a}}={\bf x}_{t_1}^{\bf a}|_{y_1=\cdots y_n=0}&=&\left({\bf x}_{t_2}^{{\bf b}}(1+\sum\limits_{0\neq {\bf v}\in\mathbb N^n,~~{\bf u}\in\mathbb Z^n}c_{\bf v}{\bf y}^{\bf v}{\bf x}_{t_2}^{\bf u})\right)|_{y_1=\cdots=y_n=0}\nonumber\\
&=&{\bf x}_{t_0}^{G_{t_2}{\bf b}}(1+\sum\limits_{0\neq {\bf v}\in\mathbb N^n,~~{\bf u}\in\mathbb Z^n}c_{\bf v}{\bf y}^{\bf v}{\bf x}_{t_0}^{G_{t_2}{\bf u}})|_{y_1=\cdots=y_n=0}\nonumber\\
&=&{\bf x}_{t_0}^{G_{t_2}{\bf b}}.\nonumber
\end{eqnarray}
Thus $G_{t_2}{\bf b}=G_{t_1}{\bf a}=G_{t_2}{\bf f}$, which implies ${\bf b}={\bf f}$, by Theorem \ref{detgthm} (ii).
 Then by (\ref{equ12}),  ${\bf x}_{t_1}^{\bf a}$ can be written as follows:
\begin{eqnarray}\label{adeqn}
{\bf x}_{t_1}^{\bf a}={\bf x}_{t_2}^{{\bf f}}+\sum\limits_{0\neq {\bf v_1}\in\mathbb N^n,~~{\bf u_1}\in\mathbb Z^n}c_{\bf v_1}{\bf y}^{\bf v_1}{\bf x}_{t_2}^{\bf u_1},\;\;\text{with}\;\; c_{\bf v_1}\geq 0.
\end{eqnarray}
Similarly, the Laurent expansion of the cluster monomial  ${\bf x}_{t_2}^{\bf f}$ with respect to ${\bf x}_{t_1}$  has the following form
$${\bf x}_{t_2}^{\bf f}={\bf x}_{t_1}^{{\bf a}}+\sum\limits_{0\neq {\bf v_2}\in\mathbb N^n,~~{\bf u_2}\in\mathbb Z^n}c_{\bf v_2}{\bf y}^{\bf v_2}{\bf x}_{t_1}^{\bf u_2},\;\;\text{with}\;\; c_{\bf v_2}\geq 0.$$
Thus we  obtain
$$\sum\limits_{0\neq {\bf v_1}\in\mathbb N^n,~~{\bf u_1}\in\mathbb Z^n}c_{\bf v_1}{\bf y}^{\bf v_1}{\bf x}_{t_2}^{\bf u_1}+\sum\limits_{0\neq {\bf v_2}\in\mathbb N^n,~~{\bf u_2}\in\mathbb Z^n}c_{\bf v_2}{\bf y}^{\bf v_2}{\bf x}_{t_1}^{\bf u_2}=0.$$
Clearly, $\sum\limits_{0\neq {\bf v_2}\in\mathbb N^n,~~{\bf u_2}\in\mathbb Z^n}c_{\bf v_2}{\bf y}^{\bf v_2}{\bf x}_{t_1}^{\bf u_2}$ can be written as follows:
$$\sum\limits_{0\neq {\bf v_2}\in\mathbb N^n,~~{\bf u_2}\in\mathbb Z^n}c_{\bf v_2}{\bf y}^{\bf v_2}{\bf x}_{t_1}^{\bf u_2}=\frac{f_1}{g_1},$$
 where $f_1,g_1\in\mathbb Z_{\geq 0}[y_1,\cdots,y_n,x_{1;t_1},\cdots,x_{n;t_1}]$.
By positive Laurent phenomenon,  ${\bf x}_{t_2}^{\bf u_1}$ can be written as a quotient of two polynomials in $\mathbb Z_{\geq 0}[y_1,\cdots,y_n,x_{1;t_1},\cdots,x_{n;t_1}]$. Because $c_{\bf v_1}\geq 0$,  we can write  $\sum\limits_{0\neq {\bf v_1}\in\mathbb N^n,~~{\bf u_1}\in\mathbb Z^n}c_{\bf v_1}{\bf y}^{\bf v_1}{\bf x}_{t_2}^{\bf u_1}$ in the following form:
$$\sum\limits_{0\neq {\bf v_1}\in\mathbb N^n,~~{\bf u_1}\in\mathbb Z^n}c_{\bf v_1}{\bf y}^{\bf v_1}{\bf x}_{t_2}^{\bf u_1}=\frac{f_2}{g_2},$$
where $f_2,g_2\in\mathbb Z_{\geq0}[y_1,\cdots,y_n,x_{1;t_1},\cdots,x_{n;t_1}]$.
Thus
$$\sum\limits_{0\neq {\bf v_1}\in\mathbb N^n,~~{\bf u_1}\in\mathbb Z^n}c_{\bf v_1}{\bf y}^{\bf v_1}{\bf x}_{t_2}^{\bf u_1}+\sum\limits_{0\neq {\bf v_2}\in\mathbb N^n,~~{\bf u_2}\in\mathbb Z^n}c_{\bf v_2}{\bf y}^{\bf v_2}{\bf x}_{t_1}^{\bf u_2}=\frac{f_2}{g_2}+\frac{f_1}{g_1}=0,$$
then we obtain $f_1g_2+f_2g_1=0$, where $f_1,g_1,f_2,g_2\in\mathbb Z_{\geq0}[y_1,\cdots,y_n,x_{1;t_1},\cdots,x_{n;t_1}]$ and $g_1\neq0,g_2\neq0$.
So, we have $f_2=0=f_1$, i.e., $$\sum\limits_{0\neq {\bf v_1}\in\mathbb N^n,~~{\bf u_1}\in\mathbb Z^n}c_{\bf v_1}{\bf y}^{\bf v_1}{\bf x}_{t_2}^{\bf u_1}=0=\sum\limits_{0\neq {\bf v_2}\in\mathbb N^n,~~{\bf u_2}\in\mathbb Z^n}c_{\bf v_2}{\bf y}^{\bf v_2}{\bf x}_{t_1}^{\bf u_2},$$
Then, by (\ref{adeqn}), we obtain that ${\bf x}_{t_1}^{\bf a}={\bf x}_{t_2}^{\bf f}$.
\end{proof}

\vspace{5mm}

\section{Scattering diagrams and the enough $g$-pairs property}

In this section, we will recall some results in \cite{M,GHKK}  on scattering diagrams of cluster algebras, then we use these results to show that any skew-symmetrizable cluster algebra $\mathcal A(\mathcal S)$  with principal coefficients at $t_0$ has the enough $g$-pairs property. In addition, we explain why the sign-coherence of $G$-matrices is a direct result of the enough $g$-pairs property.

\subsection {Scattering diagrams}
The scattering diagrams  associated with  skew-symmetrizable matrices are a main tool in \cite{GHKK}, which are used to affirm some important conjectures in cluster algebras. Now we use the reference \cite{M} to review the  scattering diagrams.

Let $R=\mathbb Q[x_1^{\pm1},\cdots,x_n^{\pm1}][[y_1,\cdots,y_n]]$ be the formal power series in the variables $y_1,\cdots,y_n$  with coefficients in $\mathbb Q[x_1^{\pm1},\cdots,x_n^{\pm1}]$. Let $B$ be a skew-symmetrizable integer matrix  and $S=diag(s_1,\cdots,s_n)$ is the skew-symmetrizer of $B$  such that the trace $tr(S)$  is minimal. Clearly, if $B$ is skew-symmetric, then $S=I_n$.

For $0\neq{\bf v}\in\mathbb N^n$, define the formal {\bf elementary transformation} $E_{\bf v}:R\rightarrow R$ by
$$E_{\bf v}({\bf x}^{\bf w})=(1+{\bf x}^{B\bf v}{\bf y}^{\bf v})^{\frac{{\bf v}^{\top}S{\bf w}}{gcd(S{\bf v})}}{\bf x}^{\bf w}, ~~E_{\bf v}({\bf y}^{{\bf w}^{\prime}})={\bf y}^{{\bf w}^{\prime}},$$
  for any ${\bf w}\in\mathbb Z^n$, ${\bf w}^\prime\in\mathbb N^n$. Here, $gcd(S{\bf v})$ means the greatest common divisor of $s_1v_1,\cdots,s_nv_n$.

$E_{\bf v}$ is an automorphism of $R$, with inverse $E_{\bf v}^{-1}$ given by

$$E_{\bf v}^{-1}({\bf x}^{\bf w})=(1+{\bf x}^{B\bf v}{\bf y}^{\bf v})^{-\frac{{\bf v}^{\top}S{\bf w}}{gcd(S{\bf v})}}{\bf x}^{\bf w}, ~~E_{\bf v}^{-1}({\bf y}^{{\bf w}^{\prime}})={\bf y}^{{\bf w}^{\prime}}.$$

For $B$, a {\bf wall} is a pair $({\bf v},W)$, where $0\neq {\bf v}\in \mathbb N^n$ and $W$ is a convex polyhedral cone in $\mathbb R^n$ which spans ${\bf v}^{\bot}:=\{{\bf m}\in \mathbb R^n|{\bf v}^\top S {\bf m}=0\}$. Here $B$ is used to provide the elementary transformation $E_{\bf v}$ for the wall  $({\bf v},W)$. The open half-space $\{{\bf m}\in \mathbb R^n|{\bf v}^\top S {\bf m}>0\}$ is called the {\bf green side of $W$}, and $\{{\bf m}\in \mathbb R^n|{\bf v}^\top S {\bf m}<0\}$ is called  the
{\bf red side of $W$}.

A {\bf scattering diagram} $\mathfrak D$ is a collection of walls in $\mathbb R^n$ for the same matrix $B$. A
smooth path $\rho:[0,1]\rightarrow\mathbb R^n$ in a scattering diagram $\mathfrak D$  is called {\bf finite transverse} if
\begin{itemize}
\item $\rho(0)$ and $\rho(1)$ are not in any walls;
\item whenever the image of $\rho$ intersects a wall $({\bf v},W)$, it crosses $W$ transversely;
\item the image of $\rho$ intersects finitely many walls, and does not intersect the boundary
of a wall or the intersection of two walls which span different hyperplanes.
\end{itemize}

Let $\rho$ be a finite transverse path, and list the walls
crossed by $\rho$ in order: $$({\bf v}_1, W_1),\cdots,({\bf v}_s, W_s).$$  Then $\rho$ determines
the {\bf path-ordered product of elementary transformations} given by
$$E_{{\bf v}_s}^{\epsilon_s}\cdots  E_{{\bf v}_2}^{\epsilon_2}\circ E_{{\bf v}_1}^{\epsilon_1}:R\rightarrow R,$$
where $\epsilon_i=1$ (resp.,  $\epsilon_i=-1$) if $\rho$ crosses $W_i$ from its green side to its red side (resp., from its red side to its green
side).

  A scattering diagram  $\mathfrak D$ with finitely many walls is called a {\bf consistent scattering diagram}  if any path-ordered product defined by finite transverse loop is the trivial automorphism of $R$.
Two scattering diagrams $\mathfrak D_1$ and $\mathfrak D_2$ with finitely many walls are called {\bf equivalent}  if any smooth path $\rho$ which is finite transitive in both  diagrams
determines the same path-ordered product.

For brevity, the definition of  consistency or equivalence for scattering diagrams with infinitely many walls will not be given here. One can refer to \cite{M} for them.

\begin{Theorem}(\cite[Theorem 1.12]{GHKK})
For each  skew-symmetrizable matrix $B$, there is an unique consistent scattering diagram
$\mathfrak D(B)$  up to equivalence, such that
\begin{itemize}
\item for each $i\in\{1,2,\cdots,n\}$, there is a wall of the form $({\bf e}_i,{\bf e}_i^{\bot})$ and
\item every other wall $({\bf v}, W)$ in $\mathfrak D(B)$ has the property that $B{\bf v}\notin W$.
\end{itemize}
\end{Theorem}

A {\bf chamber} of $\mathfrak D(B)$ is  a path-connected
component of $\mathbb R^n-\mathfrak D(B)$. Let $({\bf v}, W)$ be a wall of $\mathfrak D(B)$.
Since $0\neq {\bf v}\in \mathbb N^n$, we know that $W\cap(\mathbb R_{>0})^n=\phi=W\cap(\mathbb R_{<0})^n$.
Thus no walls can pass through $(\mathbb R_{>0})^n$ and $(\mathbb R_{<0})^n$, they must be chambers in $\mathfrak D(B)$. They are called {\bf all-positive chamber} and {\bf all-negative chamber} respectively.
 A chamber $\mathcal C$ of   $\mathfrak D(B)$ is a {\bf reachable chamber}, if there exists a finite
transverse path $\rho$ from all-positive chamber $(\mathbb R_{>0})^n$ to $\mathcal C$.

\begin{Theorem} (\cite[Lemma 2.10]{GHKK})\label{thmgvectors}
Let  $\mathcal A(\mathcal S)$ be a skew-symmetrizable cluster algebra of rank $n$ with principal coefficients at $t_0$. Then every reachable chamber of $\mathfrak D(B_{t_0})$ is of the form $$\mathbb R_{>0}{\bf g}_1+\cdots+\mathbb R_{>0}{\bf g}_n,$$
where $G=({\bf g}_1,\cdots,{\bf g}_n)$ is a certain $G$-matrix of $\mathcal A(\mathcal S)$.
\end{Theorem}

  For $I=\{i_1,\cdots,i_p\}\subseteq\{1,\cdots,n\}$, we assume that $i_1<i_2<\cdots<i_p$. Let $\pi_I: \mathbb R^n\rightarrow \mathbb R^{|I|}=\mathbb R^{p}$ be the  canonical projection given by
$\pi_I({\bf m})=(m_{i_1},\cdots,m_{i_p})^{\top}$, for ${\bf m}=(m_1,\cdots,m_n)^{\top}\in\mathbb R^n$.
 Let $\pi_I^{\top}: \mathbb R^{|I|}=\mathbb R^{p}\rightarrow \mathbb R^n$ be the coordinate inclusion given by $$\pi_I^{\top}({\bf v})_i=\begin{cases} v_{k},&\text{if }i =i_k\in I;\\0,& \text{otherwise},\end{cases}$$
 for ${\bf v}=(v_1,\cdots,v_p)^{\top}\in\mathbb R^p$.

Let $I$ be a subset of $\{1,2,\cdots,n\}$, and $B^\dag$ be the principal submatrix  of $B$ defined by $I$. Recall that, in \cite{M}, the pull back of the scattering diagram $\mathfrak D(B^\dag)$ is the scattering diagram
$$\pi_I^{\ast}\mathfrak D(B^\dag)=\{(\pi_I^{\top}( {\bf v}), \pi_I^{-1}(W))|({\bf v},W)\in \mathfrak D(B^\dag)\}.$$

In \cite[Theorem 33]{M}, Muller has described how to obtained the scattering diagram $\pi_I^{\ast}\mathfrak D(B^\dag)$ from the scattering diagram $\mathfrak D(B)$ for skew-symmetric case. In fact, this description can be naturally extended to skew-symmetrizable  version in the paralleled method, which means the following theorem.
\begin{Theorem}\label{thmpullback}
 For each  skew-symmetrizable matrix $B$, the scattering diagram $\pi_I^{\ast}\mathfrak D(B^\dag)$
is obtained from $\mathfrak D(B)$ by removing all of its walls $({\bf v}, W)$, where ${\bf v}=(v_1,\cdots,v_n)$ has any nonzero entry $v_i$ with $i\notin I$.
\end{Theorem}

\subsection{The enough $g$-pairs property}

Let $I$ be a subset of $\{1,\cdots,n\}$. We say that $(k_1,\cdots,k_s)$ is an {\bf $I$-sequence}, if $k_j\in I$ for $j=1,\cdots,s$.

\begin{Definition}
Let $\mathcal A(\mathcal S)$ be a skew-symmetrizable cluster algebra of rank $n$ with initial seed at $t_0$, and  $I=\{i_1,\cdots,i_p\}$ be a subset of $\{1,2,\cdots,n\}$.

(i) We say that a seed $({\bf x}_t,{\bf y}_t, B_t)$  of $\mathcal A(\mathcal S)$ is {\bf connected with $({\bf x}_{t_0}, {\bf y}_{t_0}, B_{t_0})$ by an $I$-sequence}, if there exists an $I$-sequence $(k_1,\cdots, k_s)$ such that $$({\bf x}_t,{\bf y}_t, B_t)=\mu_{k_s}\cdots\mu_{k_2}\mu_{k_1}({\bf x}_{t_0}, {\bf y}_{t_0}, B_{t_0}).$$

(ii) We say that a cluster ${\bf x}_t$ of $\mathcal A(\mathcal S)$ is {\bf connected with ${\bf x}_{t_0}$ by an $I$-sequence}, if there exists a seed containing the cluster ${\bf x}_t$ such that this seed is connected with
a seed containing the cluster ${\bf x}_{t_0}$ by an $I$-sequence.
\end{Definition}
Clearly, if the cluster ${\bf x}_t$ is connected with ${\bf x}_{t_0}$ by an $I$-sequence, then $x_{i;t}=x_{i;t_0}$ for $i\notin I$.

\begin{Definition}\label{maindef}
Let $\mathcal A(\mathcal S)$ be a skew-symmetrizable cluster algebra of rank $n$ with principal coefficients at $t_0$, and  $I$ be a subset of $\{1,\cdots,n\}$.

(i).  For two clusters ${\bf x}_t,{\bf x}_{t^\prime}$ of $\mathcal A(\mathcal S)$, the pair $({\bf x}_t,{\bf x}_{t^\prime})$ is called a {\bf $g$-pair along $I$}, if it satisfies the following  conditions:
\begin{itemize}
\item ${\bf x}_{t^\prime}$ is connected with ${\bf x}_{t_0}$ by an $I$-sequence and

\item for any cluster monomial ${\bf x}_t^{\bf v}$ in ${\bf x}_t$, there exists a cluster monomial ${\bf x}_{t^{\prime}}^{{\bf v}^{\prime}}$ in ${\bf x}_{t^{\prime}}$ with $v_i^{\prime}=0$ for $i\notin I$
  such that
$$\pi_I(g({\bf x}_t^{\bf v}))=\pi_I(g({\bf x}_{t^{\prime}}^{{\bf v}^{\prime}})),$$
where $g({\bf x}_t^{\bf v})$ and $g({\bf x}_{t^{\prime}}^{{\bf v}^{\prime}})$ are $g$-vectors of the cluster monomials
${\bf x}_t^{\bf v}$ and ${\bf x}_{t^{\prime}}^{{\bf v}^{\prime}}$ respectively.
\end{itemize}

(ii). $\mathcal A(\mathcal S)$ is said to have  {\bf the enough $g$-pairs property} if for  any subset $I$ of $\{1,\cdots,n\}$ and any cluster ${\bf x}_t$ of $\mathcal A(\mathcal S)$, there exists a cluster ${\bf x}_{t^\prime}$ such that $({\bf x}_t,{\bf x}_{t^\prime})$ is  a $g$-pair along $I$.
\end{Definition}

 In next proposition  we will provide another understanding of the $g$-pair along $I$, which will be used to reveal the connection between the enough $g$-pairs property and the sign-coherence of $G$-matrices.

Let $A=(a_{ij})$ be an $n\times n$ matrix, and $I$ be a subset of $[1,n]:=\{1,\cdots,n\}$ and denote by $A|_{I\times I}=(a_{ij})_{i,j\in I}$ and $A|_{I\times [1,n]}=(a_{i,j})_{i\in I,j\in [1,n]}$.
\begin{Proposition}\label{prodef}
Assume that $\mathcal A(\mathcal S)$ is a skew-symmetrizable cluster algebra of rank $n$ with principal coefficients at $t_0$. Let  ${\bf x}_t,{\bf x}_{t^\prime}$ be two clusters of $\mathcal A(\mathcal S)$, and $G_t, G_{t^\prime}$
 be their $G$-matrices respectively.
  Then  $({\bf x}_t,{\bf x}_{t^\prime})$ is a $g$-pair along $I\subseteq [1,n]$ if and only if
 ${\bf x}_{t^\prime}$ is connected with ${\bf x}_{t_0}$ by an $I$-sequence and
 $G_t|_{I\times [1,n]}=G_{t^\prime}|_{I\times I}Q$ for some $Q\in M_{|I|\times n}(\mathbb Z_{\geq 0})$.
\end{Proposition}
\begin{proof}

If $({\bf x}_t,{\bf x}_{t^\prime})$ is a $g$-pair along $I$, then we know that  ${\bf x}_{t^\prime}$ is connected with ${\bf x}_{t_0}$ by an $I$-sequence and
for the cluster variable $x_{k;t}$, there exists a cluster monomial ${\bf x}_{t^{\prime}}^{{\bf v}_k^{\prime}}$ in ${\bf x}_{t^{\prime}}$ with that the $i$-th component of ${\bf v}_k^{\prime}$ is zero  for $i\notin I$
  such that $\pi_I(g(x_{k;t}))=\pi_I(g({\bf x}_{t^{\prime}}^{{\bf v}_k^{\prime}}))=\pi_I(G_{t^\prime}{\bf v}_k^\prime)$, where $k=1,\cdots,n$. Set $Q=(\pi_I({\bf v}_1^\prime),\cdots,\pi_I({\bf v}_n^\prime))\in M_{|I|\times n}(\mathbb Z_{\geq 0})$.
  Then we get $G_t|_{I\times [1,n]}=G_{t^\prime}|_{I\times I}Q$.

  If ${\bf x}_{t^\prime}$ is connected with ${\bf x}_{t_0}$ by an $I$-sequence and
 $G_t|_{I\times [1,n]}=G_{t^\prime}|_{I\times I}Q$ for some $Q\in M_{|I|\times n}(\mathbb Z_{\geq 0})$, by the definition of $g$-pair along $I$, we only need to show that for any cluster monomial ${\bf x}_t^{\bf v}$ in ${\bf x}_t$, there exists a cluster monomial ${\bf x}_{t^{\prime}}^{{\bf v}^{\prime}}$ in ${\bf x}_{t^{\prime}}$ with $v_i^{\prime}=0$ for $i\notin I$
  such that
$$\pi_I(g({\bf x}_t^{\bf v}))=\pi_I(g({\bf x}_{t^{\prime}}^{{\bf v}^{\prime}})).$$

  Denote by ${\bf q}_k$ the $k$-th column vector of $Q$, then $\pi_I^\top({\bf q}_k)\in\mathbb N^n$ and the $i$-th component of $\pi_I^\top({\bf q}_k)$ is zero for $i\notin I$.
  It is easy to see that $G_t|_{I\times [1,n]}=G_{t^\prime}|_{I\times I}Q$  just means that
  $$\pi_I(g(x_{k;t}))=\pi_I(G_{t^\prime} \pi_I^\top({\bf q}_k))=\pi_I({\bf x}_{t^\prime}^{\pi_I^\top({\bf q}_k)}),$$
  for $k=1,\cdots,n$. For any cluster monomial  ${\bf x}_t^{\bf v}$ in ${\bf x}_t$, set ${\bf v}^\prime=v_1\pi_I^\top({\bf q}_1)+\cdots+v_n\pi_I^\top({\bf q}_n)\in\mathbb N^n$, then  $v_i=0$ for $i\notin I$ and
  \begin{eqnarray}
  \pi_I(g({\bf x}_t^{\bf v}))&=& \pi_I(v_1g(x_{1;t})+\cdots+v_ng(x_{n;t}))\nonumber\\
&=&v_1\pi_I(g(x_{1;t}))+\cdots+v_n\pi_I(g(x_{n;t}))\nonumber\\
&=&v_1\pi_I({\bf x}_{t^\prime}^{\pi_I^\top({\bf q}_1)})+\cdots+v_n\pi_I({\bf x}_{t^\prime}^{\pi_I^\top({\bf q}_n)})\nonumber\\
&=&\pi_I(v_1G_{t^\prime}\pi_I^\top({\bf q}_1)+\cdots+v_nG_{t^\prime}\pi_I^\top({\bf q}_n))\nonumber\\
   &=&\pi_I(G_{t^\prime}{\bf v}^{\prime})= \pi_I(g({\bf x}_{t^{\prime}}^{{\bf v}^{\prime}}))\nonumber.
\end{eqnarray}
\end{proof}

\begin{Corollary}\label{corsign}
If $\mathcal A(\mathcal S)$ has the enough $g$-pairs property, then the sign-coherence of $G$-matrices holds, i.e., any two nonzero entries of a $G$-matrix of $\mathcal A(\mathcal S)$ in the same row have the same sign.
\end{Corollary}
\begin{proof}
For a cluster ${\bf x}_t$ and $I=\{k\}$, since $\mathcal A(\mathcal S)$ has the enough $g$-pairs property, we know that there exists a cluster ${\bf x}_{t^{\prime}}$ such that $({\bf x}_t,{\bf x}_{t^{\prime}})$ is a $g$-pair along $I$. Then by Proposition \ref{prodef}, ${\bf x}_{t^{\prime}}$ is connected with ${\bf x}_{t_0}$ by an $I=\{k\}$-sequence and
there exists a matrix $Q\in M_{1 \times n}(\mathbb Z_{\geq 0})$ such that
$G_t|_{I\times [1,n]}=G_{t^\prime}|_{I\times I}Q$. Since  ${\bf x}_{t^{\prime}}$ is connected with ${\bf x}_{t_0}$ by an $I=\{k\}$-sequence, we know that  ${\bf x}_{t^{\prime}}$ is equal to ${\bf x}_{t_0}$ or is obtained from  ${\bf x}_{t_0}$ by one step of mutation in the direction of $k$. Thus we know that $G_{t^\prime}|_{I\times I}=G_{t^\prime}|_{\{k\}\times \{k\}}=\pm 1$ and $G_t|_{\{k\}\times [1,n]}=G_t|_{I\times [1,n]}=G_{t^\prime}|_{I\times I}Q=\pm Q$. By $Q\in M_{1 \times n}(\mathbb Z_{\geq 0})$, we know that any two nonzero entries in the $k$-th row of $G_t$ have the same sign.
\end{proof}

By Corollary \ref{corsign} and its proof, the  enough $g$-pairs property can be understood as a strong version of sign-coherence of $G$-matrices. Recently, the sign-coherence of $G$-matrices for skew-symmetrizable cluster algebras has been proved in \cite{GHKK}. The next theorem aims to show that the  strong version of sign-coherence of $G$-matrices, i.e., the enough $g$-pairs property also holds for skew-symmetrizable cluster algebras.

\begin{Theorem}(Existence theorem)\label{thmenough}
Any skew-symmetrizable cluster algebra $\mathcal A(\mathcal S)$ with principal coefficients at $t_0$ has  the enough $g$-pairs property.
\end{Theorem}
\begin{proof}
For any subset $I=\{i_1,\cdots,i_p\}$ of $\{1,2,\cdots,n\}$, and any cluster ${\bf x}_t$ of $\mathcal A(\mathcal S)$, we need to find a cluster ${\bf x}_t^\prime$, such that $({\bf x}_t,{\bf x}_t^\prime)$ is a $g$-pair along $I$. We will use scattering diagrams to find such a cluster ${\bf x}_t^\prime$.

Let $B_{t_0}^\dag$ be the principal submatrix of $B_{t_0}$ defined by $I$.
By Theorem \ref{thmpullback}, the scattering diagram $\pi_I^{\ast}\mathfrak D(B_{t_0}^\dag)$
is obtained from $\mathfrak D(B_{t_0})$ by removing all of its walls $({\bf v}, W)$, where ${\bf v}=(v_1,\cdots,v_n)$ has any nonzero entry $v_i$ with $i\notin I$. Thus we can get the following facts:\vspace{2mm}
\\
{\bf Fact (i)}£ºEvery wall $({\bf v},W)$ in $\pi_I^{\ast}\mathfrak D(B_{t_0}^\dag)$ is a wall in $\mathfrak D(B_{t_0})$;

From this fact, we have that
\\
{\bf Fact (ii)}: Each chamber $\mathcal C$ of $\mathfrak D(B_{t_0})$ is contained in some chamber ${\mathcal C}^\dag$ of  $\pi_I^{\ast}\mathfrak D(B_{t_0}^\dag)$;

From Fact (i) and  (ii), we obtain that
\\
 {\bf Fact (iii)}: Each  finite transverse path $\rho: [0,1]\rightarrow \mathbb R^n$  in $\mathfrak D(B_{t_0})$ is also a finite transverse path in $\pi_I^{\ast}\mathfrak D(B_{t_0}^\dag)$;

From Fact (ii) and (iii), we have that
\\
{\bf Fact (iv)}: Each reachable chamber $\mathcal C$ of $\mathfrak D(B_{t_0})$ is contained in some reachable chamber ${\mathcal C}^\dag$ of  $\pi_I^{\ast}\mathfrak D(B_{t_0}^\dag)$.

 From the definition of $\pi_I^{\ast}\mathfrak D(B_{t_0}^\dag)$, we have the following fact:
\\
 {\bf Fact (v)}: If  ${\mathcal C}^\dag$ is a (reachable) chamber of  $\pi_I^{\ast}\mathfrak D(B_{t_0}^\dag)$, then $\pi_I({\mathcal C}^\dag)$ is a (reachable) chamber of $\mathfrak D(B_{t_0}^\dag)$.

For any cluster ${\bf x}_t$ of  $\mathcal A(\mathcal S)$, it has a $G$-matrix $G_t$. By Theorem \ref{thmgvectors}, there exists a reachable chamber $\mathcal C_t$ corresponding to $G_t$.
By Fact (iv), there exists a reachable chamber $\mathcal C_t^{\dag}$  of $\pi_I^{\ast}\mathfrak D(B_{t_0}^\dag)$
such that  $\mathcal C_t\subseteq \mathcal C_t^{\dag}$.

By Fact (v), $\pi_I({\mathcal C_t}^\dag)$ is a reachable chamber of $\mathfrak D(B_{t_0}^\dag)$. By Theorem \ref{thmgvectors} again, there exists a $G$-matrix $G_{t\dag}=({\bf g}_1^{t\dag},\cdots,{\bf g}_p^{t\dag})$ of $B_{t_0}^\dag$ corresponding to the reachable chamber $\pi_I({\mathcal C_t}^\dag)$ of  $\mathfrak D(B_{t_0}^\dag)$.
Suppose that the $G$-matrix $G_{t\dag}$ of $B_{t_0}^\dag$ is obtained from the initial $G$-matrix of  $B_{t_0}^\dag$
by the sequence of mutations $\mu_{k_s}\cdots\mu_{k_2}\mu_{k_1}$, where $k_i\in I$ for $i=1,\cdots,s$.

Clearly, the seed $({\bf x}_{t^{\prime}},{\bf y}_{t^{\prime}},B_{t^{\prime}}):=\mu_{k_s}\cdots\mu_{k_2}\mu_{k_1}({\bf x}_{t_0},{\bf y}_{t_0},B_{t_0})$ is connected with $({\bf x}_{t_0},{\bf y}_{t_0},B_{t_0})$ by an $I$-sequence. We will show that $({\bf x}_t,{\bf x}_t^\prime)$ is a $g$-pair along $I$.

For any cluster monomial ${\bf x}_t^{\bf v}$ in ${\bf x}_t$, we need to find a cluster monomial ${\bf x}_{t^{\prime}}^{{\bf v}^{\prime}}$ in ${\bf x}_{t^{\prime}}$ with $v_i^{\prime}=0$ for $i\notin I$
  such that
$$\pi_I(g({\bf x}_t^{\bf v}))=\pi_I(g({\bf x}_{t^{\prime}}^{{\bf v}^{\prime}})).$$

 We know that the $g$-vector $g({\bf x}_t^{\bf v})$ of ${\bf x}_t^{\bf v}$ is $$v_1{\bf g}_1^t+\cdots +v_n{\bf g}_n^t\in\mathbb R_{\geq0}{\bf g}_1^t+\cdots+\mathbb R_{\geq0}{\bf g}_n^t=\overline {\mathcal C_t},$$
where $\overline{\mathcal C_t}$ is the closure of $\mathcal C_t$ in $\mathbb R^n$.
Thus $$\pi_I(g({\bf x}_t^{\bf v}))\in\pi_I(\overline{\mathcal C})\subseteq\pi_I(\overline{\mathcal C_t^{\dag}})=\overline{\pi_I(\mathcal C_t^{\dag})}=\mathbb R_{\geq 0}{\bf g}_1^{t\dag}+\cdots+\mathbb R_{\geq 0}{\bf g}_p^{t\dag},$$
i.e., we get that
$$\pi_I(g({\bf x}_t^{\bf v}))\in\mathbb R_{\geq 0}{\bf g}_1^{t\dag}+\cdots+\mathbb R_{\geq 0}{\bf g}_p^{t\dag}.$$
By $\pi_I(g({\bf x}_t^{\bf v}))\in \mathbb Z^p$ and $det (G_{t\dag})=\pm 1$ (Theorem \ref{detgthm}), we actually have that
$$\pi_I(g({\bf x}_t^{\bf v}))\in \mathbb Z_{\geq 0}{\bf g}_1^{t\dag}+\cdots+\mathbb Z_{\geq 0}{\bf g}_p^{t\dag}.$$
Thus we can assume that $\pi_I(g({\bf x}_t^{\bf v}))=v_1^{\dag}{\bf g}_1^{t\dag}+\cdots+v_p^{\dag}{\bf g}_p^{t\dag}$, where $v^{\dag}=(v_1^{\dag},\cdots,v_p^{\dag})^{\top}\in\mathbb N^p$.
 Set ${\bf v}^\prime=\pi_I^{\top}({\bf v}^{\dag})\in\mathbb N^n$, then ${\bf x}_{t^{\prime}}^{{\bf v}^{\prime}}$ is cluster monomial in ${\bf x}_{t^{\prime}}$ with $v_i^{\prime}=0$ for $i\notin I$. It is easy to see that  $\pi_I({\bf x}_{t^{\prime}}^{\pi_I^{\top}({\bf v}^{\dag})})=v_1^{\dag}{\bf g}_1^{t\dag}+\cdots+v_p^{\dag}{\bf g}_p^{t\dag}$, thus $\pi_I(g({\bf x}_t^{\bf v}))=\pi_I(g({\bf x}_{t^{\prime}}^{{\bf v}^{\prime}}))$.

 Indeed, without loss of generality, we can assume that $I=\{1,2,\cdots,p\}$. We know that the $G$-matrix of ${\bf x}_{t^{\prime}}$ has the form of $$G_{t^{\prime}}=\begin{pmatrix}G_{t\dag}&0\\\ast&I_{n-p}\end{pmatrix}.$$
Then
 $$\pi_I({\bf x}_{t^{\prime}}^{\pi_I^{\top}({\bf v}^{\dag})})=\pi_I(G_{t^\prime}\pi_I^{\top}({\bf v}^{\dag}))= G_{t\dag}{\bf v}^{\dag} =v_1^{\dag}{\bf g}_1^{t\dag}+\cdots+v_p^{\dag}{\bf g}_p^{t\dag}=\pi_I(g({\bf x}_t^{\bf v}\pi_I^{\top}({\bf v}^{\dag}))).$$
 The proof is finished.
\end{proof}
The above theorem is the existence theorem for the enough $g$-pairs, which says that for any subset $I=\{i_1,\cdots,i_p\}$ of $\{1,2,\cdots,n\}$, and any cluster ${\bf x}_t$ of $\mathcal A(\mathcal S)$, there exists a cluster ${\bf x}_t^\prime$, such that $({\bf x}_t,{\bf x}_t^\prime)$ is a $g$-pair along $I$.
In fact, such ${\bf x}_{t^{\prime}}$ is unique by uniqueness theorem given in Section 5.

\vspace{5mm}

\section{Uniqueness  of $g$-pairs}
In this section, we prove the uniqueness theorem for $g$-pairs. Before giving the uniqueness theorem, some preparations need be made.

\begin{Lemma}\label{lemvariable}
Let $\mathcal A(\mathcal S)$ be a skew-symmetrizable cluster algebra with principal coefficients at $t_0$, and ${\bf x}_t, {\bf x}_{t^\prime}$ be two clusters of $\mathcal A(\mathcal S)$.  If there exists a vector ${\bf v}=(v_1,\cdots,v_n)^\top\in\mathbb N^n$ such that $x_{1;t}={\bf x}_{t^\prime}^{\bf v}$, then $x_{1,t}$ is a cluster variable in ${\bf x}_{t^{\prime}}$.
\end{Lemma}
\begin{proof}
Because ${\bf v}\in\mathbb N^n$, we get that $v_i\geq0$ for $i=1,\cdots,n$. Since  $x_{1;t}={\bf x}_{t^\prime}^{\bf v}$,  ${\bf v}$ can not be zero, thus $v_1+v_2+\cdots+v_n\geq 1$.  In order to show that $x_{1,t}$ is a cluster variable in ${\bf x}_{t^{\prime}}$, it suffices to show that $v_1+v_2+\cdots+v_n=1$, due to  $x_{1;t}={\bf x}_{t^\prime}^{\bf v}$.

Assume that $v_1+v_2+\cdots+v_n>1$. Let $f_{j}$ be the expansion of $x_{j;t^\prime}$ with respect to ${\bf x}_t$, by Theorem \ref{detgthm} (i), we know that each $f_j$ has the form of
$$x_{j;t^{\prime}}={\bf x}_{t}^{{\bf r}_{j;t}^{t^{\prime}}}(1+\sum\limits_{0\neq {\bf v}\in\mathbb N^n,~~{\bf u}\in\mathbb Z^n}c_{\bf v}{\bf y}^{\bf v}{\bf x}_t^{\bf u}),\;\;  c_{\bf v}\geq 0.$$

  Replacing $f_1,\cdots, f_n$ into $x_{1;t}={\bf x}_{t^\prime}^{\bf v}$, we obtain that
  \begin{eqnarray}\label{eqnff}
 x_{1;t}=f_{1}^{v_1}\cdots f_{n}^{v_n}.
 \end{eqnarray}
 Because the left side of the above equality does not contain any one of $y_1,\cdots,y_n$, we know that each $f_j$ is a Laurent monomial of the form $f_j={\bf x}_{t}^{{\bf r}_{j;t}^{t^{\prime}}}$. Since the left side of the equality (\ref{eqnff}) is a monomial with total exponent $1$ and $v_1+\cdots+v_n>1$, we get that  at least one of $f_{1},\cdots, f_{n}$ is a proper Laurent monomial.
Say,  $f_{i_0}$  is a proper Laurent monomial, written as
 \begin{eqnarray}\label{eqn123}
 x_{i_0;t^{\prime}}=f_{i_0}={\bf x}_t^{\bf m}\;\;\;\text{with }m_l<0.
 \end{eqnarray}

 Let  $({\bf x}_{t_1},{\bf y}_{t_1}, B_{t_1})=\mu_{l}({\bf x}_{t},{\bf y}_{t}, B_{t})$, then $x_{i;t_1}=x_{i;t}$ for $i\neq l$ and $x_{l;t_1}x_{l;t}$ is a binomial in $x_{1;t_1},\cdots,x_{n;t_1}$. By replacing this into the right-side of (\ref{eqn123}), we get the expansion of $x_{i_0;t^{\prime}}$ with respect to ${\bf x}_{t_1}$. Since $m_l<0$, this obtained expansion  contains a binomial in its denominator, so it  cannot be a Laurent polynomial. This contradicts to the Laurent phenomenon. So $v_1+v_2+\cdots+v_n=1$, then we get that only one of $v_1,\cdots,v_n$ is $1$, and the others are zero. Say $v_k=1$. Thus  $x_{1;t}={\bf x}_{t^\prime}^{\bf v}$ is just $x_{1;t}=x_{k;t^\prime}\in {\bf x}_{t^{\prime}}$. The proof is finished.
\end{proof}

If we have a $g$-pair  $({\bf x}_t,{\bf x}_{t^\prime})$ along $I=\{1,\cdots,n-1\}$, the following lemma tells that the $d$-vectors of cluster variables in ${\bf x}_t$ with respect to ${\bf x}_{t^\prime}$ enjoying nice property. We will generalize this idea in Theorem \ref{thmcompatible}.

\begin{Lemma}\label{lemdv}
 Supposed that $\mathcal A(\mathcal S)$ is a skew-symmetrizable cluster algebra of rank $n$ with principal coefficients at $t_0$ and $({\bf x}_t,{\bf x}_{t^\prime})$  is a $g$-pair along $I=\{1,\cdots,n\}\backslash \{k\}$. Let $G_t, G_t^{\prime}$ be the $G$-matrices of  $({\bf x}_t,{\bf x}_{t^\prime})$ respectively, and denote by $R_{t}^{t^\prime}=(r_{jl}):=G_{t^{\prime}}^{-1}G_t$. Let $d^{t^\prime}(x_{i;t})=(d_{1}^\prime,\cdots,d_{n}^\prime)$ be the $d$-vector of $x_{i;t}$ with respect to ${\bf x}_{t^\prime}$. We have that

 (i). The $j$th row vector of $R_{t}^{t^\prime}$ is a non-negative vector for any $j\neq k$.

(ii). $r_{ki}>0$ if and only if $d_{k}^\prime=-1$,  and  if only if $x_{i;t}\in{\bf x}_{t^\prime}$ and $x_{i;t}=x_{k;t^\prime}$.

(iii). $r_{ki}=0$ if and only if $d_{k}^\prime=0$,  and if only if $x_{i;t}\in{\bf x}_{t^\prime}$ and $x_{i;t}\neq x_{k;t^\prime}$.

(iv). $r_{ki}<0$ if and only if $d_{k}^\prime>0$,  and if only if $x_{i;t}\notin{\bf x}_{t^\prime}$.

In particular, $d_{k}^\prime=-1,0$ or  $d_{k}^\prime>0$.
\end{Lemma}

\begin{proof}
Without loss of generality, we can assume that $k=n$.

(i):  We consider the Laurent expansion of $x_{i;t}$ with respect to ${\bf x}_{t^{\prime}}$, say $x_{i;t}=f({\bf x}_{t^\prime})$. By Theorem \ref{detgthm}, we know that the Laurent monomial ${\bf x}_{t^\prime}^{{\bf r}}$ appears in $f({\bf x}_{t^\prime})$, where ${\bf r}$ is the unique vector in $\mathbb Z^n$ such that $g(x_{i;t})=(g_{1}^t,\cdots,g_{n-1}^t,g_{n}^t)^{\top}=G_{t^\prime}{\bf r}$. Thus ${\bf r}=G_{t^\prime}^{-1}g(x_{i;t})=G_{t^{\prime}}^{-1}G_t{\bf e_i}$ is just the $i$th column vector of $R_{t}^{t^\prime}$.

Since $({\bf x}_t,{\bf x}_t^\prime)$ is a $g$-pair along $I=\{1,\cdots,n-1\}$, we know that for the cluster variable $x_{i;t}$ (as a cluster monomial in ${\bf x}_t$), there exists a cluster monomial ${\bf x}_{t^\prime}^{{\bf v}^\prime}$, with $v_j^\prime=0$ is zero for $j\notin I$, i.e., $v_n^\prime=0$ such that $$\pi_I(g(x_{i;t}))=\pi_I(g({\bf x}_{t^{\prime}}^{{\bf v}^\prime}))=\pi_I(G_{t^\prime}{\bf v}^\prime).$$

Since ${\bf x}_{t^\prime}$ is connected with ${\bf x}_{t_0}$ by an $I=\{1,\cdots,n-1\}$-sequence, we know that the $G$-matrix of ${\bf x}_{t^\prime}$ has the form of $G_{t^{\prime}}=\begin{pmatrix}G(t^{\prime})&0\\ \ast& 1\end{pmatrix}$.
Thus  $\pi_I(g(x_{i;t}))=\pi_I(g({\bf x}_{t^{\prime}}^{{\bf v}^\prime}))=\pi_I(G_{t^\prime}{\bf v}^\prime)$ just means that
$$(g_{1}^t,\cdots,g_{n-1}^t)^{\top}=G(t^\prime)(v_{1}^\prime,\cdots,v_{n-1}^\prime)^{\top}.$$

 Since $g(x_{i;t})=(g_{1}^t,\cdots,g_{n-1}^t,g_{n}^t)^{\top}=G_{t^\prime}{\bf r}$, we get that  $$(g_{1}^t,\cdots,g_{n-1}^t)^{\top}=G(t^\prime)(r_{1},\cdots,r_{n-1})^{\top}.$$

By Theorem \ref{detgthm}, we know that $det(G(t^\prime))=det(G_{t^{\prime}})=\pm1$. Thus we get that $$(r_{1},\cdots,r_{n-1})^{\top}=(v_{1}^\prime,\cdots,v_{n-1}^\prime)^{\top}\in\mathbb N^{n-1}.$$

Since $i$ is arbitrary, we can get that the $j$th row  vector of $R_{t}^{t^\prime}$ is a non-negative vector for any $j\neq n$.

(ii), (iii): If $r_{n}\geq 0$, then ${\bf r}\in \mathbb N^n$, and ${\bf x}_{t^\prime}^{{\bf r}}$ is a cluster monomial in ${\bf x}_{t^\prime}$ having the same $g$-vector  with the cluster variable $x_{i;t}$. By Theorem \ref{thmmonomial}, we get that $x_{i;t}={\bf x}_{t^\prime}^{{\bf r}}$. Then by Lemma \ref{lemvariable}, $x_{i;t}$ is a cluster variable in ${\bf x}_{t^\prime}$. More precisely, if $r_{n}>0$, then $x_{i;t}=x_{n;t^{\prime}}$, and thus $d_{n}^\prime=-1$. If $r_{n}=0$, then $x_{i;t}=x_{k;t^\prime}$ for some $k\neq n$. In this case, $d_{n}^\prime=0$.

(iv): If $r_{n}<0$, then the exponent of $x_{n;t^\prime}$ in  ${\bf x}_{t^\prime}^{{\bf r}}$ is negative, and thus the minimal exponent of $x_{n;t^\prime}$ appearing in $f({\bf x}_{t^\prime})$ is negative. Recall that the meaning of $-d_{n}^\prime$ is just the minimal exponent of $x_{n;t^\prime}$ appearing in $f({\bf x}_{t^\prime})$. So $d_{n}^\prime$ is positive. The proof is finished.
\end{proof}

\begin{Corollary}\label{thmdvectors}
Let $\mathcal A(\mathcal S)$ be a skew-symmetrizable cluster algebra of rank $n$ with principal coefficients at $t_0$, $x$ be a cluster variable of $\mathcal A(\mathcal S)$ and $d^{t_0}(x)=(d_1,\cdots,d_n)^{\top}$ be the $d$-vector of $x$ with respect to ${\bf x}_{t_0}$. Then $d_k=-1,0$ or $d_k>0$ for each $k=1,2,\cdots,n$ and

(i). if $d_k=-1$, then $x=x_{k;t_0}$;

(ii). if $d_k=0$, then there exists a cluster ${\bf x}_{t^{\prime}}$ containing both $x$ and $x_{k;t_0}$.
\end{Corollary}
\begin{proof}
 Let ${\bf x}_t$ be a cluster containing the cluster variable $x$. By Theorem \ref{thmenough}, we know that $\mathcal A(\mathcal S)$ has  the enough $g$-pairs property. Thus for $I=\{1,2,\cdots,n\}\backslash\{k\}$ and the cluster ${\bf x}_t$, there exists a cluster ${\bf x}_{t^\prime}$ such that $({\bf x}_t,{\bf x}_{t^\prime})$ is a $g$-pair along $I=\{1,2,\cdots,n\}\backslash\{k\}$. Since ${\bf x}_{t^\prime}$  is connected with ${\bf x}_{t_0}$ by an $I=\{1,2,\cdots,n\}\backslash\{k\}$-sequence, we know that $x_{k;t^\prime}=x_{k;t_0}$ and  $d_k=d_k^\prime$, by Proposition \ref{prodvectors}.
 Then the result follows  from $x_{k;t^\prime}=x_{k;t_0}$ and Lemma \ref{lemdv}.
\end{proof}

Let ${\bf x}_t^{\bf v}$ be a cluster monomial of a cluster algebra,  the set $Su({\bf x}_t^{\bf v}):=\{x_{i;t}| v_i>0\}$ is called the {\bf support} of the cluster monomial ${\bf x}_t^{\bf v}$, and the cluster variables in this support are called the {\bf support cluster variables} for ${\bf x}_t^{\bf v}$.
\begin{Proposition}\label{prosupport}
Let  $\mathcal A(\mathcal S)$ be a skew-symmetrizable cluster algebra  with principal coefficients at $t_0$, and ${\bf x}_{t_1}^{\bf a},{\bf x}_{t_2}^{\bf f}$ be two cluster monomials of $\mathcal A(\mathcal S)$. If  ${\bf x}_{t_1}^{\bf a}={\bf x}_{t_2}^{\bf f}$, then

(i). ${\bf x}_{t_1}^{\bf a}$ and ${\bf x}_{t_2}^{\bf f}$ have the same support, i.e., $Su({\bf x}_{t_1}^{\bf a})=Su({\bf x}_{t_2}^{\bf f})$;

(ii). $x_{i;t_1}=x_{j;t_2}$ implies $a_i=f_j$.
\end{Proposition}

\begin{proof}
(i). ${\bf x}_{t_1}^{\bf a}={\bf x}_{t_2}^{\bf f}$ can be viewed as the Laurent expansion of the cluster monomial ${\bf x}_{t_1}^{\bf a}$ with respect to the cluster ${\bf x}_{t_2}$. By this expansion, we know the $d$-vector of the cluster monomial ${\bf x}_{t_1}^{\bf a}$ is the vector $-{\bf f}\in \mathbb Z_{\leq 0}^n$. By Corollary \ref{thmdvectors}, for each $-f_j<0$, $x_{j;t_2}$ must be a cluster variable in support of ${\bf x}_{t_1}^{\bf a}$, i.e., each support cluster variable $x_{j;t_2}$ of ${\bf x}_{t_2}^{\bf f}$ is a support cluster variable of ${\bf x}_{t_1}^{\bf a}$.

Similarly, we can show each support cluster variable $x_{i;t_1}$ of ${\bf x}_{t_1}^{\bf f}$ is a support cluster variable of ${\bf x}_{t_2}^{\bf f}$, thus  ${\bf x}_{t_1}^{\bf a}$ and ${\bf x}_{t_2}^{\bf f}$ have the same support.

(ii). If $x_{i;t_1}$ is not in the support of the two cluster monomials, then $a_i=f_j=0$.
If $x_{i;t_1}$ is a support cluster variable of the two cluster monomials, then $a_i=f_j$ follows from the algebraical independence of  cluster variables in the same cluster.
\end{proof}

\begin{Theorem}(Uniqueness theorem)\label{thmunique}
Let  $\mathcal A(\mathcal S)$ be a skew-symmetrizable cluster algebra  with principal coefficients at $t_0$, and $({\bf x}_t,{\bf x}_{t_1})$, $({\bf x}_t,{\bf x}_{t_2})$ be two $g$-pairs along $I$, then ${\bf x}_{t_1}={\bf x}_{t_2}$.
\end{Theorem}
\begin{proof}
Without loss of generality, we can assume $I=\{1,2,\cdots,p\}$. Since both clusters ${\bf x}_{t_1},{\bf x}_{t_2}$ are connected with ${\bf x}_{t_0}$ by an $I$-sequence, we know that their $G$-matrices have the following form.
$$G_{t_1}=\begin{pmatrix}G(t_1)&0\\ W_1 &I_{n-p}\end{pmatrix}, G_{t_2}=\begin{pmatrix}G(t_2)&0\\ W_2&I_{n-p}\end{pmatrix},$$
where  $W_1, W_2$ are nonnegative matrices, by sign coherence of $G$-matrices.
Because $({\bf x}_t,{\bf x}_{t_1})$ and $({\bf x}_t,{\bf x}_{t_2})$ are $g$-pairs along $I$ and by Proposition \ref{prodef}, there exist $Q_1,Q_2\in M_{p\times n}(\mathbb Z_{\geq0})$ such that $$G(t_1)Q_1=G_{t_1}|_{I\times I}Q_1=G_t|_{I\times [1,n]}=G_{t_2}|_{I\times I}Q_2=G(t_2)Q_2.$$

It is easy to see that
 $$G_{t_1}\begin{pmatrix}Q_1\\W_2 Q_2\end{pmatrix}=\begin{pmatrix}G(t_1)&0\\ W_1 &I_{n-p}\end{pmatrix} \begin{pmatrix}Q_1\\W_2 Q_2\end{pmatrix}=\begin{pmatrix}G(t_2)&0\\ W_2&I_{n-p}\end{pmatrix} \begin{pmatrix}Q_2\\W_1 Q_1\end{pmatrix} =G_{t_2}\begin{pmatrix}Q_2\\W_1 Q_1\end{pmatrix}.$$
Note that $A=(a_{ij})_{n\times n}:=\begin{pmatrix}Q_1\\W_2 Q_2\end{pmatrix}$ and $F=(f_{ij})_{n\times n}:=\begin{pmatrix}Q_2\\W_1 Q_1\end{pmatrix}$ are non-negative matrices.
So the cluster monomials ${\bf x}_{t_1}^{{\bf a}_i}$ and ${\bf x}_{t_2}^{{\bf f}_i}$ have the same $g$-vectors for $i=1,\cdots,n$, where ${\bf a}_i, {\bf f}_i$ are $i$th column vectors of $A$ and $F$. By Theorem \ref{thmmonomial} and Proposition \ref{prosupport}, we get that ${\bf x}_{t_1}^{{\bf a}_i}$ and ${\bf x}_{t_2}^{{\bf f}_i}$ are equal and $Su({\bf x}_{t_1}^{{\bf a}_i})=Su({\bf x}_{t_2}^{{\bf f}_i})$ $i=1,\cdots,n$. Thus $\bigcup\limits_{i=1}^n Su({\bf x}_{t_1}^{{\bf a}_i})=\bigcup\limits_{i=1}^n Su({\bf x}_{t_2}^{{\bf f}_i})$.

By Theorem \ref{detgthm},  we get that  $G_t|_{I\times [1,n]}$ is row full rank, so $Q_1$ and $Q_2$ do not zero rows. Thus $\{x_{1;t_1},\cdots,x_{p;t_1}\}\subseteq \bigcup\limits_{i=1}^n Su({\bf x}_{t_1}^{{\bf a}_i})\subseteq {\bf x}_{t_1}$ and $\{x_{1;t_2},\cdots,x_{p;t_2}\}\subseteq \bigcup\limits_{i=1}^n Su({\bf x}_{t_2}^{{\bf f}_i})\subseteq {\bf x}_{t_2}$. By $\bigcup\limits_{i=1}^n Su({\bf x}_{t_1}^{{\bf a}_i})=\bigcup\limits_{i=1}^n Su({\bf x}_{t_2}^{{\bf f}_i})$, we know that $\{x_{1;t_1},\cdots,x_{p;t_1}\}\subseteq {\bf x}_{t_2}$ and $\{x_{1;t_2},\cdots,x_{p;t_2}\}\subseteq {\bf x}_{t_1}$.

Since both clusters ${\bf x}_{t_1},{\bf x}_{t_2}$ are connected with ${\bf x}_{t_0}$ by an $I$-sequence, we know that $x_{k;t_1}=x_{k;t_0}=x_{k;t_2}$ for $k=p+1,\cdots,n$. Thus we must have ${\bf x}_{t_1}={\bf x}_{t_2}$.
\end{proof}
\vspace{5mm}

\section{ Answers to
 Conjecture \ref{conjecture} and Conjecture \ref{conjecture2}}

In this section, we will give positive answers to Conjecture \ref{conjecture} and Conjecture \ref{conjecture2}, then we explain that why $d$-vectors are interesting.

\subsection{Answer to Conjecture \ref{conjecture2}}

 We start this subsection with a proposition which says the coefficients in cluster algebras do not affect the combinatorial structure of clusters of cluster algebras (see Proposition \ref{proclustervariable} for details). So when the properties we focus on are mainly on the combinatorial structure of clusters of cluster algebras, we can reduce to the case of cluster algebras with principal coefficients. Thus we can make full use of the enough $g$-pairs property.

\begin{Proposition}\label{proclustervariable}
Let $\mathcal A(\mathcal S(1)), \mathcal A(\mathcal S(2))$ be two skew-symmetrizable cluster algebras having the same exchange matrix at $t_0$.  Denoted by $({\bf x}_t(k),{\bf y}_t(k), B_t(k))$, the seed of $\mathcal A(\mathcal S(k))$ at $t\in\mathbb T_n$, $k=1,2$. The following hold.

 (i). $x_{i;t_1}(1)=x_{j;t_2}(1)$ if and only if $x_{i;t_1}(2)=x_{j;t_2}(2)$, where $t_1,t_2\in\mathbb T_n$ and $i,j\in\{1,2,\cdots,n\}$.

 (ii). If there exists a permutation $\sigma$ of $\{1,\cdots,n\}$ such that $x_{i;t_1}(1)=x_{\sigma(i);t_2}(1)$ for $i=1,\cdots,n$, then $y_{i;t_1}(1)=y_{\sigma(i);t_2}(1)$ and $b_{ij}^{t_1}(1)=b_{\sigma(i)\sigma(j)}^{t_2}(1)$ for any $i$ and $j$.
\end{Proposition}
\begin{proof}
(i). Since $B_{t_0}(1)=B_{t_0}(2)$, we know that $B_{t}(1)=B_{t}(2)$ for any $t\in\mathbb T_n$.
Let $\mathcal A(\mathcal S^{pr})$ be the skew-symmetrizable cluster algebra with principal coefficients at $t_1$ and with initial exchange matrix $B_{t_1}(1)=B_{t_1}(2)$.
The seed of $\mathcal A(\mathcal S^{pr})$ at $t\in\mathbb T_n$ is denoted by $({\bf x}_t^{pr},{\bf y}_t^{pr},B_{t}^{pr})$.

If $x_{i;t_1}(1)=x_{j;t_2}(1)$, we know that $d^{t_1}(x_{j;t_2}(1))=-{\bf e}_i$, where ${\bf e}_i$ is the $i$-th column vector of $I_n$. Since the $d$-vectors do not depend on the choice of coefficients, we know that $d^{t_1}(x_{j;t_2}^{pr})=-{\bf e}_i$. Then by Corollary \ref{thmdvectors}, we have that $x_{i;t_1}^{pr}=x_{j;t_2}^{pr}$.
Thus we obtain that $x_{i;t_1}(2)=x_{j;t_2}(2)$ from Theorem 3.7 of \cite{FZ3}.
By the same argument, we can show that if $x_{i;t_1}(2)=x_{j;t_2}(2)$, then $x_{i;t_1}(1)=x_{j;t_2}(1)$.

(ii). By (i), we know if there exists a permutation $\sigma$ of $\{1,\cdots,n\}$ such that $x_{i;t_1}(1)=x_{\sigma(i);t_2}(1)$ for $i=1,\cdots,n$, then $x_{i;t_1}^{pr}=x_{\sigma(i);t_2}^{pr}$ for any $i$.
Then by \cite[Theorem 4]{GSV}, we have $y_{i;t_1}^{pr}=y_{\sigma(i);t_2}^{pr}$ and $b_{ij}^{t_1}(pr)=b_{\sigma(i)\sigma(j)}^{t_2}(pr)$. Then the result follows from \cite[Theorem 4.6]{FZ3}.
\end{proof}

The following theorem is an affirmation to  Conjecture \ref{conjecture2}.

\begin{Theorem}\label{corconnected}
For any skew-symmetrizable cluster algebra $\mathcal A(\mathcal S)$,  the seeds of $\mathcal A(\mathcal S)$ whose clusters contain particular cluster variables form a connected subgraph of the exchange graph of $\mathcal A(\mathcal S)$.
\end{Theorem}
\begin{proof}
For the cluster variables $x_{p+1;t_0},\cdots,x_{n;t_0}$ of the cluster ${\bf x}_{t_0}$, we consider the seeds of $\mathcal A(\mathcal S)$ whose clusters contain these variables. We need to check that if these seeds form a connected subgraph of the exchange graph of $\mathcal A(\mathcal S)$.

Thanks to Proposition \ref{proclustervariable}, we can assume that $\mathcal A(\mathcal S)$ is a cluster algebra with principal coefficients at $t_0$. Let ${\bf x}_t$ be a cluster containing the variables  $x_{p+1;t_0},\cdots,x_{n;t_0}$.
 Without  loss of generality, we can assume that $x_{i;t}=x_{i;t_0}$ for $i=p+1,\cdots,n$.

  By Theorem \ref{thmenough}, we know that $\mathcal A(\mathcal S)$ has  the enough $g$-pairs property. Thus for $I=\{1,2,\cdots, p\}$ and the cluster ${\bf x}_t$, there exists a cluster ${\bf x}_{t^\prime}$ such that $({\bf x}_t,{\bf x}_{t^\prime})$ is a $g$-pair along $I$. We will show that ${\bf x}_t={\bf x}_t^\prime$, thus ${\bf x}_t$ is connected with ${\bf x}_{t_0}$ by a $\{1,2,\cdots,p\}$-sequence.

 Since $({\bf x}_t,{\bf x}_{t^\prime})$ is a $g$-pair along $I$, we know that for the cluster variable $x_{k;t}$ (as a cluster monomial in ${\bf x}_t$), there exists a cluster monomial ${\bf x}_{t^{\prime}}^{{\bf v}}$ in ${\bf x}_{t^{\prime}}$ with $v_i=0$  for $i\notin I$ such that
$$\pi_I(g(x_{k;t}))=\pi_I(g({\bf x}_{t^{\prime}}^{{\bf v}}))=\pi_I(G_{t^{\prime}}{\bf v}).$$

We know that the $G$-matrix of ${\bf x}_{t^\prime}$ has the form of $G_{t^{\prime}}=\begin{pmatrix}G(t^{\prime})&0\\ \ast& I_{n-p}\end{pmatrix}$. By the assumption, the $G$-matrix of ${\bf x}_t$ also has the form of $G_{t}=\begin{pmatrix}G(t)&0\\ \star& I_{n-p}\end{pmatrix}$.
By row sign-coherence of $G$-matrices, we know that both $\ast$ and $\star$ are nonnegative matrices.
Because $\pi_I(g(x_{k;t}))=\pi_I(g({\bf x}_{t^{\prime}}^{{\bf v}}))=\pi_I(G_{t^{\prime}}{\bf v})$, we can assume that
\begin{eqnarray}
g(x_{k;t})&=&(g_1,\cdots,g_p,g_{p+1},\cdots,g_n)^{\top};\nonumber\\
g({\bf x}_{t^{\prime}}^{{\bf v}})&=&G_{t^{\prime}}{\bf v}=(g_1,\cdots,g_p,g_{p+1}^{\prime},\cdots,g_n^{\prime})^{\top},\nonumber
\end{eqnarray}
where $g_i,g_i^{\prime}\geq 0$ for $i=p+1,\cdots,n$. Then it is easy to see that $x_{k;t}\prod\limits_{i=p+1}^nx_{i;t}^{g_i^{\prime}}$ and ${\bf x}_{t^{\prime}}^{{\bf v}}\prod\limits_{i=p+1}^nx_{i;t^{\prime}}^{g_i}$ are cluster monomials in ${\bf x}_t$ and in ${\bf x}_{t^{\prime}}$ respectively and
\begin{eqnarray}
g(x_{k;t}\prod\limits_{i=p+1}^nx_{i;t}^{g_i^{\prime}})=(g_1,\cdots,g_p,g_{p+1}+g_{p+1}^{\prime},
\cdots,g_{n}+g_{n}^{\prime})^{\top}=g({\bf x}_{t^{\prime}}^{{\bf v}}\prod\limits_{i=p+1}^nx_{i;t^{\prime}}^{g_i}).\nonumber
\end{eqnarray}
This means that the cluster monomials $x_{k;t}\prod\limits_{i=p+1}^nx_{i;t}^{g_i^{\prime}}$ and ${\bf x}_{t^{\prime}}^{{\bf v}}\prod\limits_{i=p+1}^nx_{i;t^{\prime}}^{g_i}$ have the same $g$-vector. By Theorem \ref{thmmonomial}, we get that $$x_{k;t}\prod\limits_{i=p+1}^nx_{i;t}^{g_i^{\prime}}={\bf x}_{t^{\prime}}^{{\bf v}}\prod\limits_{i=p+1}^nx_{i;t^{\prime}}^{g_i}.$$

By Proposition \ref{prosupport}, we know that $x_{k;t}$ is a cluster variable in ${\bf x}_{t^\prime}$.
 Since $k$ is arbitrary, we obtain that ${\bf x}_{t}$ and ${\bf x}_{t^{\prime}}$ have the same cluster variables. So ${\bf x}_t$
is connected with ${\bf x}_{t_0}$ by a $\{1,2,\cdots,p\}$-sequence. By Proposition \ref{proclustervariable} (ii) or \cite[Theorem 4]{GSV}, each seed of $\mathcal A(\mathcal S)$ is determined by its cluster, we know that the seed containing ${\bf x}_{t}$ and the seed containing ${\bf x}_{t_0}$ are connected by a $\{1,\cdots,p\}$-sequence and each seed appearing in this sequence containing the cluster variables $x_{p+1;t_0},\cdots,x_{n;t_0}$. So the seeds of $\mathcal A(\mathcal S)$ whose clusters contain particular cluster variables form a connected subgraph of the exchange graph of $\mathcal A(\mathcal S)$.
\end{proof}

\subsection{Answer to Conjecture \ref{conjecture}}

The following theorem gives the affirmation to Conjecture \ref{conjecture}.

\begin{Theorem}\label{thmanswer}

Let  $\mathcal A(\mathcal S)$ is a skew-symmetrizable cluster algebra, and $d^{t_0}(x_{i;t}) = (d_1,\cdots,d_n)^{\top}$ be the $d$-vector of the cluster variable $x_{i;t}$ with respect to the cluster ${\bf x}_{t_0}$ of $\mathcal A(\mathcal S)$. Then

(i).  $d_k$ depends only on $x_{i;t}$ and $x_{k;t_0}$, not on the clusters containing $x_{k;t_0}$, where $k=1,\cdots,n$;

(ii).  $d_k\geq -1$ for $k=1,\cdots,n$, and in details,
\begin{eqnarray}
d_k=\begin{cases} -1~,& \text{iff}\;\; x_{i,t}=x_{k,t_0};\\ 0~,& \text{iff}\;\;x_{i,t}\not=x_{k,t_0}\;\;\text{and}\;\; x_{i,t}, x_{k,t_0}\in {\bf x}_{t^\prime} \; \text{for some}\; t^\prime;\\ \text{a positive integer}~, & \text{iff}\;\; \text{there exists no cluster } {\bf x}_{t^\prime} \text{ containing both }x_{i,t} \text{ and }x_{k,t_0} .\end{cases}\nonumber
\end{eqnarray}
In particular, if $x_{i;t}\notin{\bf x}_{t_0}$, then $d^{t_0}(x_{i;t})\in\mathbb N^n$, that is, the positivity of $d$-vectors holds.

\end{Theorem}

\begin{proof}

(i). Let ${\bf x}_{t^\prime}$ be another cluster of $\mathcal A(\mathcal S)$ containing $x_{k;t_0}$, and $d^{t^\prime}(x_{i;t})=(d_1^\prime,\cdots,d_n^\prime)^{\top}$ be the $d$-vector of $x_{i;t}$ with respect to ${\bf x}_{t^\prime}$.
 Without loss of generality, we can assume that $x_{k;t^\prime}=x_{k;t_0}$. By Theorem \ref{corconnected}, the cluster ${\bf x}_{t^\prime}$ is connected with ${\bf x}_{t_0}$ by a $\{1,2,\cdots,n\}\setminus\{k\}$-sequence.
 Then by Proposition \ref{prodvectors}, we get that $d_k^\prime=d_k$.

(ii). $d_k\geq -1$ follows from Corollary  \ref{thmdvectors} directly. The important thing is when $d_k$ takes values at $-1,0$, a positive integer.

 By Proposition \ref{prodrec}, we know that the notion of $d$-vectors is independent of the choice of coefficient system. Thanks to Proposition \ref{proclustervariable}, we can assume that  $\mathcal A(\mathcal S)$ is a skew-symmetrizable cluster algebra with principal coefficients at $t_0$.

(a). If  $x_{i;t}=x_{k;t_0}$, clearly, $d_k=-1$. Conversely, if $d_k=-1$, then $x_{i;t}=x_{k;t_0}$ follows from  Corollary \ref{thmdvectors} (i).

(b). By Corollary  \ref{thmdvectors}, if $d_k=0$, then there exists a cluster ${\bf x}_{t^{\prime}}$ containing both $x_{i;t}$ and $x_{k;t_0}$. Since $d_k=0$, we must have $x_{i;t}\neq x_{k;t_0}$.

Conversely, suppose that $x_{i;t}\neq x_{k;t_0}$ and there exists a cluster ${\bf x}_{t^{\prime}}$ containing both $x_{i;t}$ and $x_{k;t_0}$. Without loss of generality, we can assume that $x_{i;t}=x_{1;t^\prime}$ and $x_{k;t_0}=x_{2;t^\prime}$.
By (i), we know that $d_k$ does not depend on the cluster containing $x_{k;t_0}=x_{2;t^\prime}$, so $d_k$ is equal to the $2$-th component of $d^{t^\prime}(x_{i;t})$. Since $d^{t^\prime}(x_{i;t})=d^{t^\prime}(x_{1;t^{\prime}})=-{\bf e}_1$, we get $d_k=0$.

(c). By (a) and (b), $d_k\leq 0$ if and only if there exists a cluster containing both $x_{i,t}$ and $x_{k,t_0}$. So $d_k>0$ if and only if there exists no cluster containing both $x_{i,t}$ and $x_{k,t_0}$.

(d). We have proved that $d_k\geq -1$, and if $d_k=-1$, then $x_{i;t}\in {\bf x}_{t_0}$.
So if $x_{i;t}\notin{\bf x}_{t_0}$, then each $d_k$ is nonnegative, and we thus $d^{t_0}(x_{i;t})\in\mathbb N^n$.
The proof is finished.
\end{proof}

 As a corollary, we can get the sign-coherence of $D$-matrices, which was conjectured in \cite{FZ3}.
\begin{Corollary}
Suppose that $\mathcal A(\mathcal S)$ is a skew-symmetrizable cluster algebra with coefficients in $\mathbb {ZP}$, and ${\bf x}_t,{\bf x}_{t_0}$ be two clusters of $\mathcal A(\mathcal S)$. Let $D_t^{t_0}$ be the $D$-matrix of ${\bf x}_t$ with respect to the cluster ${\bf x}_{t_0}$, then any two nonzero entries of $D_t$ in the same row or in the same column have the same sign.
\end{Corollary}

\subsection{Positivity, $d$-vectors and linear independence of cluster monomials} In this part, we give our reason why $d$-vectors are interesting from our perspective. Roughly speaking, the positivity of $d$-vectors can build a bridge between the positivity of cluster variables and the linear independence of cluster monomials.

Fix a cluster ${\bf x}_t$ and a vector ${\bf a}\in\mathbb Z^n$, recall that if ${\bf a}\in\mathbb N^n$,  then ${\bf x}_t^{\bf a}$ is called a {\em cluster monomial} in ${\bf x}_t$ and if ${\bf a}\in\mathbb Z^n\backslash \mathbb N^n$, then ${\bf x}_t^{\bf a}$ is called a {\em proper Laurent monomial} in ${\bf x}_t$.

 Denote by $CM(t)$  the set of cluster monomials in ${\bf x}_t$. It is easy to see that if ${\bf x}_{t_1}$ and ${\bf x}_{t_2}$ have common cluster variables, then $CM(t_1)\cap CM(t_2)\not=\emptyset$.

\begin{Definition}(\cite{CLF,CKLP})\label{deflaurent}
 A cluster algebra $\mathcal A(\mathcal S)$  is said to have  the {\bf proper Laurent monomial property} if
for any two clusters ${\bf x}_t$ and ${\bf x}_{t_0}$ of $\mathcal A(\mathcal S)$ and any cluster monomial ${\bf x}_t^{\bf a}\in CM(t)\setminus CM(t_0)$, ${\bf x}_{t}^{\bf a}$ is a $\mathbb {ZP}$-linear combination of proper Laurent monomials in ${\bf x}_{t_0}$.
\end{Definition}

The proper Laurent monomial property is a very interesting property for cluster algebra, which can be seen from the following theorem.

\begin{Theorem}\label{lemproper}
 (\cite{CLF}) If a cluster algebra $\mathcal A(\mathcal S)$ has the proper Laurent monomial property, then its cluster monomials are linearly independent over $\mathbb {ZP}$.
\end{Theorem}
\begin{proof}
In order to make the readers to have a good sense of the proper Laurent monomial property, we repeat the proof  here.

Suppose that $\sum\limits_{t, {\bf v}} c_{t,{\bf v}}{\bf x}_t^{\bf v}=0$, where ${\bf v}\in\mathbb N^n, c_{t,{\bf v}}\in\mathbb {ZP}$. Fix a cluster ${\bf x}_{t_0}$ and a ${\bf v}_0\in \mathbb N^n$, by the proper Laurent monomial property,  each ${\bf x}_t^{\bf v}\notin CM(t_0)$ is a sum of proper Laurent monomials in $\mathbb{ZP}[x_{1;t_0}^{\pm1},\cdots,x_{n;t_0}^{\pm1}]$. Thus
$\sum\limits_{t, {\bf v}} c_{t,{\bf v}}{\bf x}_t^{\bf v}$ can be written as a Laurent polynomial in $\mathbb{ZP}[x_{1;t_0}^{\pm1},\cdots,x_{n;t_0}^{\pm1}]$ with the form of $\Sigma_1+\Sigma_2$,
where $\Sigma_1$ is a sum of monomials in $\mathbb{ZP}[x_{1;t_0}^{\pm1},\cdots,x_{n;t_0}^{\pm1}]$ and $\Sigma_2$ is a sum of proper Laurent monomials in $\mathbb{ZP}[x_{1;t_0}^{\pm1},\cdots,x_{n;t_0}^{\pm1}]$.
The coefficient of ${\bf x}_{t_0}^{{\bf v}_0}$  in $\Sigma_1$, as well as in $\Sigma_1+\Sigma_2$, is precisely $c_{t_0,{\bf v}_0}$.
  Then, $\Sigma_1+\Sigma_2=0$ results in $c_{t_0,{\bf v}_0}=0$.
 And thus cluster monomials of $\mathcal A(\mathcal S)$ are linearly independent
over $\mathbb {ZP}$.
\end{proof}

Recall that the {\em positivity of $d$-vectors} means that the $d$-vectors of non-initial cluster variables are non-negative vectors. Now we give the main result in this subsection, which builds a bridge between the positivity of cluster variables and the linear independence of cluster monomials via the positivity of $d$-vectors. It is very interesting for us to understand the linear independence of cluster monomials from this perspective.

\begin{Proposition}\label{thmproper}
The positivity of cluster variables and the positivity of $d$-vectors of a cluster algebra $\mathcal A(\mathcal S)$ imply the linear linear independence of cluster monomials of $\mathcal A(\mathcal S)$.
\end{Proposition}

\begin{proof}
By Theorem \ref{lemproper}, it suffices to show the positivity of cluster variables and the positivity of $d$-vectors imply the proper Laurent monomial property.

Let ${\bf x}_t$, ${\bf x}_{t_0}$ be any two clusters of $\mathcal A(\mathcal S)$, and ${\bf x}_t^{\bf a}=\prod\limits_{i=1}^{n}x_{i;t}^{a_i}$ be a cluster monomial in $CM(t)\setminus CM(t_0)$, i.e.,  there exists $a_k>0$ such that $x_{k;t}$
is not a cluster variable in ${\bf x}_{t_0}$.

By the positivity of cluster variables, the expansion of ${\bf x}_t^{\bf a}$ with respect to ${\bf x}_{t_0}$ has the form of
\begin{eqnarray}
{\bf x}_t^{\bf a}=\sum\limits_{{\bf v}\in V} c_{\bf v}{\bf x}_{t_0}^{\bf v},\nonumber
 \end{eqnarray}
 where $V$ is a finite subset of $\mathbb Z^n$, and $0\neq c_{\bf v}\in\mathbb {NP}$ for  ${\bf v}\in V$.
  For the same reason, there exist polynomials $f_1,\cdots,f_n\in \mathbb {NP}[x_{1;t},\cdots,x_{n;t}]$ with $x_{j;t}\nmid f_i$ such that
 $$x_{i;t_0}=\frac{f_i(x_{1;t},\cdots,x_{n;t})}{{\bf x}_t^{d^t(x_{i;t_0})}}.$$
Denote by $F^{\bf v}=f_1^{v_1}\cdots f_n^{v_n}$.  Since $v_1,\cdots,v_n\in \mathbb Z$, $F^{\bf v}$ can be written in  the form of $F^{\bf v}=\frac{F_{1;{\bf v}}}{F_{2;{\bf v}}}$, where
$F_{1;{\bf v}}, F_{2;{\bf v}}\in \mathbb {NP}[x_{1;t},\cdots,x_{n;t}]$  with $x_{j;t}\nmid F_{1;{\bf v}}$ and $x_{j;t}\nmid F_{2;{\bf v}}$, $j=1,2,\cdots,n$.
 Thus, $${\bf x}_t^{\bf a}=\sum\limits_{{\bf v}\in V} c_{\bf v}{\bf x}_{t_0}^{\bf v}=
\sum\limits_{{\bf v}\in V} c_{\bf v}{\bf x}_{t}^{-D_{t_0}^{t}\bf v}F^{\bf v}=\sum\limits_{{\bf v}\in V} c_{\bf v}{\bf x}_{t}^{-D_{t_0}^{t}\bf v}\frac{F_{1;{\bf v}}}{F_{2;{\bf v}}}.$$
 From the above equality, we can obtain an equality with the following form:
 $$
{\bf x}_t^{\bf a}g(x_{1;t},\cdots,x_{n;t})=\sum\limits_{{\bf v}\in V} c_{\bf v}{\bf x}_{t}^{-D_{t_0}^{t}{\bf v}}g_{\bf v}(x_{1;t},\cdots,x_{n;t}),
$$
which can be written as
$$g(x_{1;t},\cdots,x_{n;t})=\sum\limits_{{\bf v}\in V} c_{\bf v}{\bf x}_{t}^{-D_{t_0}^{t}{\bf v}-{\bf a}}g_{\bf v}(x_{1;t},\cdots,x_{n;t}).$$

Because $c_{\bf v}\in\mathbb {NP}$, and $g,~g_{\bf v}\in \mathbb {NP}[x_{1;t},\cdots,x_{n;t}]$ with $x_{j;t}\nmid g$ and $x_{j;t}\nmid g_{\bf v}$ for $j=1,2,\cdots,n$, we must have $-D_{t_0}^{t}{\bf v}-{\bf a}\in\mathbb N^n$. So, the $k$-th component of  $D_{t_0}^{t}{\bf v}+{\bf a}$ satisfying $$(d_{k1}^{t_0}v_1+\cdots+d_{kn}^{t_0}v_n)+a_k\leq0,$$ where $d_{kj}^{t_0}$ is the $k$-th component of $d^t(x_{j;t_0})$, $j=1,2,\cdots,n$. By the positivity of $d$-vectors, we get that $d_{kj}^{t_0}\geq 0$, since $x_{k;t}\notin {\bf x}_{t_0}$. Then by $a_k>0$ and $(d_{k1}^{t_0}v_1+\cdots+d_{kn}^{t_0}v_n)+a_k\leq0$, we can obtain ${\bf v}\notin \mathbb N^n, {\bf v}\in V$, i.e.,
${\bf x}_t^{\bf a}$ is a sum of proper Laurent monomials in ${\bf x}_{t_0}$. So the  proper Laurent monomial property holds for $\mathcal A(\mathcal S)$ and thus the cluster monomials of $\mathcal A(\mathcal S)$ are linear independence.
\end{proof}

\begin{Remark}\label{rmknew}
(i).  Recently, the linear independence of cluster monomials for skew-symmetrizable cluster algebras has been proved in \cite{GHKK}.

(ii).  In Proposition \ref{thmproper}, we actually provide a new proof of the linear independence of cluster monomials  for skew-symmetrizable cluster algebras, since
\\
(1). the positivity of $d$-vectors for any skew-symmetrizable cluster algebras are affirmed  in this paper, see Theorem \ref{thmanswer};
\\
(2).  the positivity of cluster variables was proved for skew-symmetric cluster algebras in \cite{LS}, for skew-symmetrizable cluster algebras in \cite{GHKK}.
\end{Remark}

\vspace{5mm}

\section{Compatibility degree on the set of cluster variables}
The compatibility degree on the set of cluster variables was firstly defined in \cite{CP} for cluster algebras of finite type. It needs Conjecture \ref{conjecture} (ii), which was affirmed for cluster algebras of finite type in \cite{CP},  to guarantee the validity of the definition of compatibility degree. Thanks to Theorem \ref{thmanswer}, we can extend the definition of  compatibility degree to any skew-symmetrizable cluster algebras.  We then give an answer to Problem \ref{problem} (see Theorem \ref{thmcompatible}).

Note that for finite type the compatibility degree defined on the set of cluster variables are almost the same with the compatibility degree defined on the set of almost positive roots (refer to \cite{FZ0,FZ1}) for corresponding root system, by \cite[Theorem 3.1]{CP}.

\begin{Definition}\label{defdegree}
Let $\mathcal A(\mathcal S)$ be a skew-symmetrizable cluster algebra, and $\mathcal X$ be the set of cluster variables of  $\mathcal A(\mathcal S)$. We define a function $d:\mathcal X\times\mathcal X\rightarrow\mathbb Z_{\geq -1}$, which is called the {\bf compatibility degree} on the set $\mathcal X$ of cluster variables. For any two cluster variables $x_{i;t}$ and $x_{j;t_0}$, $d(x_{j;t_0},x_{i;t})$ is defined by the following steps:
\begin{itemize}
\item choose a cluster ${\bf x}_{t_0}$ containing the cluster variable $x_{j;t_0}$;
\item compute the $d$-vector of $x_{i;t}$ with respect to ${\bf x}_{t_0}$, say, $d^{t_0}(x_{i;t})=(d_1,\cdots,d_n)^{\top}$;
    \item $d(x_{j;t_0},x_{i;t}):=d_j$, which is called the {\bf compatibility degree of $x_{i;t}$ with respect to $x_{j;t_0}$}.
\end{itemize}
\end{Definition}
\begin{Remark}
(1). The compatibility degree $d(x_{j;t_0},x_{i;t})$ is well-defined, since $d(x_{j;t_0},x_{i;t})$ does not depend on the choice of ${\bf x}_{t_0}$, and is uniquely determined by $x_{i;t}$ and $x_{j;t_0}$ by Theorem \ref{thmanswer} (i).

(2). By  Theorem \ref{thmanswer} (ii), we have the following facts:

(a).  $d(x_{j;t_0},x_{i;t})\in \mathbb{Z}_{\geq -1}$.

(b). $d(x_{j;t_0},x_{i;t})=-1$ if and only if $d(x_{i;t},x_{j;t_0})=-1$.

(c).  $d(x_{j;t_0},x_{i;t})=0$ if and only if $d(x_{i;t},x_{j;t_0})=0$.

(3). By (b),(c) and Theorem \ref{thmanswer} (ii), we know that $d(x_{j;t_0},x_{i;t})\leq0$ if and only if $d(x_{i;t},x_{j;t_0})\leq0$.
\end{Remark}

 We say that $x_{i;t}$ and $x_{j;t_0}$ are {\bf compatible}, if $d(x_{j;t_0},x_{i;t})\leq0$.
Let  $x_1,\cdots,x_p$ be different cluster variables of a skew-symmetrizable cluster algebra $\mathcal A(\mathcal S)$.  The set  $\{x_1,\cdots,x_p\}$ is called a {\bf compatible set} of $\mathcal A(\mathcal S)$ if any two cluster variables in this set are compatible.

 \begin{Lemma}\label{lemcompatible}
 Let $\mathcal A(\mathcal S)$ be a skew-symmetrizable cluster algebra, $x_{k;t}$ be a cluster variable, and ${\bf x}_{t_0}=\{x_{1;t_0},\cdots,x_{n;t_0}\}$ be a cluster of  $\mathcal A(\mathcal S)$. If $x_{k;t}$ is compatible with $x_{i;t_0}$ for $i=p+1,p+2,\cdots,n$, then $\{x_{k;t}\}\cup \{x_{p+1;t_0},x_{p+2;t_0},\cdots,x_{n;t_0}\}$ is a subset of some cluster of $\mathcal A(\mathcal S)$.
 \end{Lemma}

 \begin{proof}
 Thanks to Proposition \ref{proclustervariable}, we can assume that $\mathcal A(\mathcal S)$ is a  skew-symmetrizable cluster algebra with principal coefficients at $t_0$. By Theorem \ref{thmenough}, $\mathcal A(\mathcal S)$ has the enough $g$-pair property, thus for $I=\{1,2,\cdots,p\}$ and the cluster ${\bf x}_t$, there exists a cluster ${\bf x}_{t^\prime}$ such that $({\bf x}_t,{\bf x}_{t^\prime})$ is a $g$-pair along $I$. Thus for
 the cluster variable $x_{k;t}$ (as a cluster monomial in ${\bf x}_t$), there exists a  cluster monomial ${\bf x}_{t^\prime}^{\bf v}$ in ${\bf x}_{t^\prime}$ with $v_i=0$ for $i\notin I$ such that $\pi_I(g(x_{k;t}))=\pi_I(g({\bf x}_{t^\prime}^{\bf v}))$.
 We can assume that $g(x_{k;t})=(g_1,\cdots,g_p,g_{p+1},\cdots,g_n)^{\top}$ and $g({\bf x}_{t^\prime}^{\bf v})=(g_1,\cdots,g_p,g_{p+1}^\prime,\cdots,g_n^\prime)^{\top}$.
  Because  $({\bf x}_t,{\bf x}_{t^\prime})$ is a $g$-pair along $I$, we know that ${\bf x}_{t^\prime}$ is a cluster connected with ${\bf x}_{t_0}$ by an $I$-sequence, and the $G$-matrix of ${\bf x}_{t^\prime}$ has the form of $G_{t^{\prime}}=\begin{pmatrix}G(t^{\prime})&0\\ \ast& I_{n-p}\end{pmatrix}$. By row sign-coherence of $G$-matrices, we know that $g_{p+1}^\prime,\cdots,g_n^\prime\geq 0$.

 For each $i=p+1,p+2,\cdots,n$, since $x_{k;t}$ is compatible with $x_{i;t_0}$, there exists a cluster ${\bf x}_{t_i}$ containing both $x_{k;t}$ and $x_{i;t_0}$. By  the row sign-coherence of $G$-matrices, we know that the $i$-th row of $G_{t_i}$ is nonnegative, $i=p+1,\cdots,n$. Thus we can obtain that $g_{p+1},\cdots,g_n\geq 0$.

 Then  $x_{k;t}\prod\limits_{i=p+1}^nx_{i;t}^{g_i^{\prime}}$ and ${\bf x}_{t^{\prime}}^{{\bf v}}\prod\limits_{i=p+1}^nx_{i;t^{\prime}}^{g_i}$ are cluster monomials in ${\bf x}_t$ and in ${\bf x}_{t^{\prime}}$ respectively and
\begin{eqnarray}
g(x_{k;t}\prod\limits_{i=p+1}^nx_{i;t}^{g_i^{\prime}})=(g_1,\cdots,g_p,g_{p+1}+g_{p+1}^{\prime},
\cdots,g_{n}+g_{n}^{\prime})^{\top}=g({\bf x}_{t^{\prime}}^{{\bf v}}\prod\limits_{i=p+1}^nx_{i;t^{\prime}}^{g_i}).\nonumber
\end{eqnarray}
This means that the cluster monomials $x_{k;t}\prod\limits_{i=p+1}^nx_{i;t}^{g_i^{\prime}}$ and ${\bf x}_{t^{\prime}}^{{\bf v}}\prod\limits_{i=p+1}^nx_{i;t^{\prime}}^{g_i}$ have the same $g$-vector. By Theorem \ref{thmmonomial}, we get that
$$x_{k;t}\prod\limits_{i=p+1}^nx_{i;t}^{g_i^{\prime}}={\bf x}_{t^{\prime}}^{{\bf v}}\prod\limits_{i=p+1}^nx_{i;t^{\prime}}^{g_i}.$$

 By Proposition \ref{prosupport}, we get that $x_{k;t}$ is a cluster variable in ${\bf x}_{t^\prime}$. Since ${\bf x}_{t^\prime}$ is connected with ${\bf x}_{t_0}$ by an $I=\{1,\cdots,p\}$-sequence, we know that $x_{j;t^\prime}=x_{j;t_0}$ for $j=p+1,\cdots,n$. Thus ${\bf x}_{t^\prime}$ is a cluster of $\mathcal A(\mathcal S)$ containing $\{x_{k;t}\}\cup\{x_{p+1;t_0},x_{p+2;t_0},\cdots,x_{n;t_0}\}$.
 \end{proof}

The following theorem is an answer to {\em Problem} \ref{problem}.

\begin{Theorem}\label{thmcompatible}
Let $\mathcal A(\mathcal S)$ be a skew-symmetrizable cluster algebra of rank $n$. Then

(i).  $\{x_1,\cdots,x_p\}$ is  a compatible set of $\mathcal A(\mathcal S)$ if and only if  $\{x_1,\cdots,x_p\}$ is a subset of some cluster ${\bf x}_t$ of $\mathcal A(\mathcal S)$.  In this case, it always holds that $p\leq n$.

(ii). $\{x_1,\cdots,x_p\}$ is a  maximal compatible set of $\mathcal A(\mathcal S)$  if and only if it is a cluster of $\mathcal A(\mathcal S)$. In this case, $p=n$.
\end{Theorem}
\begin{proof}
(i). The side for ``if'' part is clear and the side for ``only if'' part can be deuced from Lemma \ref{lemcompatible} step by step.

(ii). It follows from (i).
\end{proof}

The following corollary gives an positive answer to \cite[Conjecture 5.5]{FST} by Fomin,  Shapiro, and  Thurston.
\begin{Corollary}
Let $\mathcal A(\mathcal S)$ be a skew-symmetrizable cluster algebra. For any collection of cluster variables of $\mathcal A(\mathcal S)$, if each pair of them is contained in some cluster of $\mathcal A(\mathcal S)$, then there exists a cluster containing all of them.
\end{Corollary}
\begin{proof}
Since  each pair of these cluster variable is contained in some cluster, we know these cluster variable forms a compatible set of  $\mathcal A(\mathcal S)$. Then the result follows from Theorem \ref{thmcompatible}.
\end{proof}

Based on Theorem \ref{thmcompatible}, we can also give a new proof of the known fact  from \cite{GSV,CL} in cluster algebras as follows.
\begin{Corollary}
Let $\mathcal A(\mathcal S)$ be a skew-symmetrizable cluster algebra of rank $n$, $({\bf x}_{t_1},{\bf y}_{t_1}, B_{t_1})$ and  $({\bf x}_{t_2},{\bf y}_{t_2}, B_{t_2})$ be two seeds of  $\mathcal A(\mathcal S)$
with $x_{i;t_1}=x_{i;t_2}$ for $i=1,2,\cdots,n-1$. Denote by $x_{n;t_1}^\prime$  the new cluster variable in the seed $\mu_n({\bf x}_{t_1},{\bf y}_{t_1}, B_{t_1})$, then $x_{n;t_2}=x_{n;t_1}$ or $x_{n;t_2}=x_{n;t_1}^\prime$.
\end{Corollary}
\begin{proof}
By the definition of mutation, we have that

$$x_{n;t_1}x_{n;t_1}^\prime=\frac{y_{n;t_1}}{1\oplus y_{n;t_1}}\prod\limits_{b_{in}^{t_1}>0}x_{i;t_1}^{b_{in}^{t_1}}+\frac{y_{n;t_1}}{1\oplus y_{n;t_1}}\prod\limits_{b_{in}^{t_1}<0}x_{i;t_1}^{-b_{in}^{t_1}}.$$
  Since $x_{i;t_1}=x_{i;t_2}$ for $i=1,2,\cdots,n-1$, we know that $\frac{y_{n;t_1}}{1\oplus y_{n;t_1}}\prod\limits_{b_{in}^{t_1}>0}x_{i;t_1}^{b_{in}^{t_1}}+\frac{y_{n;t_1}}{1\oplus y_{n;t_1}}\prod\limits_{b_{in}^{t_1}<0}x_{i;t_1}^{-b_{in}^{t_1}}$ as a polynomial in variables in ${\bf x}_{t_2}$, its $d$-vector with respect to ${\bf x}_{t_2}$ is zero, i.e., $$d^{t_2}(x_{n;t_1}x_{n;t_1}^\prime)=d^{t_2}(\frac{y_{n;t_1}}{1\oplus y_{n;t_1}}\prod\limits_{b_{in}^{t_1}>0}x_{i;t_1}^{b_{in}^{t_1}}+\frac{y_{n;t_1}}{1\oplus y_{n;t_1}}\prod\limits_{b_{in}^{t_1}<0}x_{i;t_1}^{-b_{in}^{t_1}})=0.$$
In particular, the $n$-th component of $d^{t_2}(x_{n;t_1}x_{n;t_1}^\prime)$ is zero, i.e., $d(x_{n;t_2},x_{n;t_1})+d(x_{n;t_2},x_{n;t_1}^\prime)=0.$
So, $d(x_{n;t_2},x_{n;t_1})\leq 0$ or $d(x_{n;t_2},x_{n;t_1}^\prime)\leq 0$, i.e., either that $x_{n;t_2}$ and $x_{n;t_1}$ are compatible or that   $x_{n;t_2}$ and $x_{n;t_1}^\prime$ are compatible. Assume that $x_{n;t_2}\neq x_{n;t_1}$ and $x_{n;t_2}\neq x_{n;t_1}^\prime$. We know that either $$\{x_{1;t_1},\cdots,x_{n-1;t_1},x_{n;t_2},x_{n;t_1}\}=\{x_{1;t_2},\cdots,x_{n-1;t_2},x_{n;t_2},x_{n;t_1}\}$$ forms a compatible set with $n+1$ elements, or $$\{x_{1;t_1},\cdots,x_{n-1;t_1},x_{n;t_2},x_{n;t_1}^\prime\}=\{x_{1;t_2},\cdots,x_{n-1;t_2},x_{n;t_2},x_{n;t_1}^\prime\}$$
forms a compatible set with $n+1$ elements. This  contradicts to Theorem \ref{thmcompatible},  so we must have $x_{n;t_2}=x_{n;t_1}$ or $x_{n;t_2}=x_{n;t_1}^\prime$.
\end{proof}

\begin{Example}
Let $\mathcal A(\mathcal S)$ be the coefficients-free cluster algebra given by the initial exchange matrix $B_{t_0}=\begin{pmatrix}0&1\\-1&0\end{pmatrix}$. $\mathcal A(\mathcal S)$ has five clusters:
$$\{x_1,x_2\}\xrightarrow{\mu_1}\{x_3,x_2\}\xrightarrow{\mu_2}\{x_3,x_4\}\xrightarrow{\mu_1}\{x_5,x_4\}
\xrightarrow{\mu_2}\{x_5,x_1\}\xrightarrow{\mu_1}\{x_2,x_1\},$$
where $x_3=\frac{x_2+1}{x_1},x_4=\frac{x_1+x_2+1}{x_1x_2},x_5=\frac{x_1+1}{x_2}$ and the set of cluster variables of $\mathcal A(\mathcal S)$ is $\mathcal X=\{x_1,x_2,x_3,x_4,x_5\}$. By simple computation, it is easy to see that  compatibility degree on the set $\mathcal X$ of cluster variables is given by

$$\begin{pmatrix}d(x_1,x_1)&d(x_1,x_2)&d(x_1,x_3) &d(x_1,x_4)&d(x_1,x_5)\\d(x_2,x_1)&d(x_2,x_2)&d(x_2,x_3) &d(x_2,x_4)&d(x_2,x_5)\\d(x_3,x_1)&d(x_3,x_2)&d(x_3,x_3) &d(x_3,x_4)&d(x_3,x_5)\\d(x_4,x_1)&d(x_4,x_2)&d(x_4,x_3) &d(x_4,x_4)&d(x_4,x_5)\\d(x_5,x_1)&d(x_5,x_2)&d(x_5,x_3) &d(x_5,x_4)&d(x_5,x_5) \end{pmatrix}=\begin{pmatrix}-1&0&1&1&0\\0&-1&0&1&1\\1&0&-1&0&1\\1&1&0&-1&0\\0&1&1&0&-1
\end{pmatrix}.$$
From the above matrix, we know that the maximal compatible sets of $\mathcal A(\mathcal S)$ are the following five sets:
$$\{x_1,x_2\},\{ x_3,x_2\},\{x_3,x_4 \},\{x_5,x_4\},\{x_5,x_1\},$$
which are exactly the clusters of  $\mathcal A(\mathcal S)$.
\end{Example}
\vspace{5mm}

{\bf Acknowledgements:}\; {\em This project is supported by the National Natural Science Foundation of China (No.11671350 and No.11571173).}


\end{document}